\documentclass[12pt]{amsart}
\usepackage{ amsmath, amsthm, amsfonts, amssymb, color}
 \usepackage{mathrsfs}
\usepackage{amsfonts, amsmath}
 \usepackage{amsmath,amstext,amsthm,amssymb,amsxtra}
 \usepackage{txfonts} 
 \usepackage[colorlinks, citecolor=blue,pagebackref,hypertexnames=false]{hyperref}
 \allowdisplaybreaks
 \usepackage{pgf,tikz}
\begin{document}
\baselineskip 16pt

\newcommand\C{{\mathbb C}}
\newtheorem{theorem}{Theorem}[section]
\newtheorem{proposition}[theorem]{Proposition}
\newtheorem{lemma}[theorem]{Lemma}
\newtheorem{corollary}[theorem]{Corollary}
\newtheorem{remark}[theorem]{Remark}
\newtheorem{example}[theorem]{Example}
\newtheorem{question}[theorem]{Question}
\newtheorem{exercise}[theorem]{Exercise}
\newtheorem{definition}[theorem]{Definition}
\newtheorem{conjecture}[theorem]{Conjecture}
\newcommand\RR{\mathbb{R}}
\newcommand{\la}{\lambda}
\def\RN {\mathbb{R}^n}
\newcommand{\norm}[1]{\left\Vert#1\right\Vert}
\newcommand{\abs}[1]{\left\vert#1\right\vert}
\newcommand{\set}[1]{\left\{#1\right\}}
\newcommand{\Real}{\mathbb{R}}
\newcommand{\supp}{\operatorname{supp}}
\newcommand{\card}{\operatorname{card}}
\renewcommand{\L}{\mathcal{L}}
\renewcommand{\P}{\mathcal{P}}
\newcommand{\T}{\mathcal{T}}
\newcommand{\A}{\mathbb{A}}
\newcommand{\K}{\mathcal{K}}
\renewcommand{\S}{\mathcal{S}}
\newcommand{\blue}[1]{\textcolor{blue}{#1}}
\newcommand{\red}[1]{\textcolor{red}{#1}}
\newcommand{\Id}{\operatorname{I}}
\newcommand\wrt{\,{\rm d}}
\def\SH{\sqrt {H}}

\newcommand{\rn}{\mathbb R^n}
\newcommand{\de}{\delta}
\newcommand{\tf}{\tfrac}
\newcommand{\ep}{\epsilon}
\newcommand{\vp}{\varphi}

\newcommand{\mar}[1]{{\marginpar{\sffamily{\scriptsize
        #1}}}}

\newcommand\CC{\mathbb{C}}
\newcommand\dd {\mathrm{d}}
\newcommand\NN{\mathbb{N}}
\newcommand\ZZ{\mathbb{Z}}
\renewcommand\Re{\operatorname{Re}}
\renewcommand\Im{\operatorname{Im}}
\newcommand{\mc}{\mathcal}
\newcommand\D{\mathcal{D}}
\newcommand{\al}{\alpha}
\newcommand{\nf}{\infty}
\newcommand{\comment}[1]{\vskip.3cm
	\fbox{%
		\color{red}
		\parbox{0.93\linewidth}{\footnotesize #1}}
	\vskip.3cm}

\newcommand{\disappear}[1]

\numberwithin{equation}{section}
\newcommand{\chg}[1]{{\color{red}{#1}}}
\newcommand{\note}[1]{{\color{green}{#1}}}
\newcommand{\later}[1]{{\color{blue}{#1}}}
\newcommand{\bchi}{ {\chi}}

\numberwithin{equation}{section}
\newcommand\relphantom[1]{\mathrel{\phantom{#1}}}
\newcommand\ve{\varepsilon}  \newcommand\tve{t_{\varepsilon}}
\newcommand\vf{\varphi}      \newcommand\yvf{y_{\varphi}}
\newcommand\bfE{\mathbf{E}}
\newcommand{\ale}{\text{a.e. }}

 \newcommand{\mm}{\mathbf m}
\newcommand{\Be}{\begin{equation}}
\newcommand{\Ee}{\end{equation}}

 \textwidth =162mm
\textheight =228mm
\oddsidemargin=-0.0cm
\evensidemargin=0.0cm
\headheight=13pt
\headsep=0.8cm
\parskip=0pt
\hfuzz=6pt
\widowpenalty=10000
\setlength{\topmargin}{-0.6cm}

\title[     Maximal functions generated by   H\"ormander-type   spectral multipliers]
{On maximal functions generated by   H\"ormander-type \\  spectral multipliers  }
\author[Peng Chen, Xixi Lin, Liangchuan Wu and Lixin Yan]{Peng Chen, Xixi Lin, Liangchuan Wu and Lixin Yan}

  \address{Peng Chen, Department of Mathematics, Sun Yat-sen
 	University, Guangzhou, 510275, P.R. China}
 \email{chenpeng3@mail.sysu.edu.cn}

   \address{Xixi Lin, Department of Mathematics, Sun Yat-sen   University,
 	Guangzhou, 510275, P.R. China}
 \email{linxx85@mail.sysu.edu.cn}

   \address{Liangchuan Wu, School of Mathematical Science, Anhui University,
   Hefei, 230601, P.R. China}
 \email{wuliangchuan@ahu.edu.cn}

 \address{Lixin Yan, Department of Mathematics, Sun Yat-sen   University,
 	Guangzhou, 510275, P.R. China}
 \email{mcsylx@mail.sysu.edu.cn}

 \date{\today}

\begin{abstract}

Let $(X,d,\mu)$ be a metric space with doubling measure and $L$ be a nonnegative self-adjoint operator on $L^2(X)$ whose heat kernel satisfies the Gaussian upper bound. We assume that there exists an $L$-harmonic function $h$ such that the semigroup $\exp(-tL)$,   after applying the Doob transform related to $h$, satisfies the upper and lower Gaussian estimates. In this paper
 we apply the Doob
 transform and some techniques as in  Grafakos-Honz\'ik-Seeger \cite{GHS2006}  to obtain
   an optimal  $\sqrt{\log(1+N)}$ bound in $L^p$ for the maximal function $\sup_{1\leq i\leq N}|m_i(L)f|$  for   multipliers $m_i,1\leq i\leq N,$ with   uniform estimates.  Based on this, we establish sufficient conditions on the bounded Borel function $m$ such that the maximal function 
  $M_{m,L}f(x) = \sup_{t>0} |m(tL)f(x)|$
  is bounded on $L^p(X)$.

  The applications include  Schr\"odinger operators with inverse square potential, Scattering operators,     Bessel operators and Laplace-Beltrami operators.
\end{abstract}

\maketitle
\section{Introduction}\label{sec:1}
\setcounter{equation}{0}

\subsection{Background}

Given a symbol $m$ satisfying
\begin{eqnarray}\label{e1.1}
|\partial^{\alpha} m(\xi)|\leq C_{\alpha} |\xi|^{-\alpha}
\end{eqnarray}
for all multiindices $\alpha$, then by classical Calder\'on-Zygmund theory the Fourier multiplier  $f\to {\mathcal F}^{-1}[m {\hat f}]$ defines an $L^p$ bounded operator (see for example, \cite{H1960,St1970}).

The  maximal function generated by   dilations of a single multiplier is defined by
\begin{eqnarray*}
	{\mathcal M}_mf(x)=\sup_{t>0}  \big|{\mathcal F}^{-1}\big[m(t\cdot) {\hat f}\big](x)\big|.
\end{eqnarray*}
In 2005, Christ,   Grafakos,  Honz\'ik and Seeger \cite{CGHS2005} constructed   a function $m$ satisfying \eqref{e1.1}, but  operator
${\mathcal M}_m$  is  unbounded  on $L^p({\mathbb R^n})$ for any  $1<p<\infty$.
  The counterexample shows that in general additional conditions on $m$ are needed for the maximal inequality to hold.  For example, it is seen from the results  of Dappa and Trebels \cite{DT1985}    that by  using Calder\'on-Zygmund theory,      ${\mathcal M}_m$
  is   bounded  on $L^p({\mathbb R^n})$  whenever
\begin{eqnarray}\label{e1.02}
 \sum_{k=1}^{\infty}  \|\phi m(2^k \cdot)\|^2_{W_{{n\over p}+\epsilon }^p}<\infty
\end{eqnarray}
 for $1<p\leq 2$; and
 \begin{eqnarray}\label{e1.03}
 	\sum_{k=1}^{\infty}  \|\phi m(2^k \cdot)\|^2_{W_{{n\over 2}+\epsilon }^2 }<\infty
 \end{eqnarray}
for $  2\leq p<\infty$. Here $\|F\|_{W^p_s}=\|\big(I-\Delta\big)^{s/2} F\|_{L^p(\RR^n)}$ and
$\phi \in C_c^\infty(\RR^n)$ is a non-trivial auxiliary function.
In particular,  ${\mathcal M}_m$
is   bounded  on $L^p({\mathbb R^n})$  when $\|\phi m(2^k \cdot)\|_{W^1_{{n }+\epsilon} ({\mathbb R^n})} = O(|k|^{-\alpha})
$ for some $\alpha>1/2$, which was  further weakened to the condition
$\|\phi m(2^k \cdot)\|_{W^1_{{n }+\epsilon } ({\mathbb R^n})} \leq C\big( 1+\log|k|\big)^{-(1+\epsilon')/2}
$ for some $\epsilon, \epsilon'>0$ due to Christ-Grafakos-Honz\'ik-Seeger   \cite[Theorem 1.2]{CGHS2005}.
In 2006,
    Grafakos, Honz\'ik and Seeger \cite{GHS2006} improved  the results in \cite[Theorem 1.2]{CGHS2005}
    to show the following:
    \begin{proposition}\label{prop00}
    	Suppose $1<p<\infty, q=\min\{p, 2\}$.    Suppose that
    	\begin{eqnarray*}
    		\|\phi m(2^k\cdot)\|_{W^{ q}_\alpha}\leq \omega(k), \ \ \  k\in{\mathbb Z}
    	\end{eqnarray*}
    	holds for some $\alpha>n/p + 1/p'$ if $1<p\leq 2$ or for some
    	$\alpha>n/2 + 1/p $ if $p> 2$,   and the nonincreasing rearrangement $\omega^{\ast}$ satisfies
    	\begin{eqnarray*}
    		\omega^{\ast}(0) +\sum_{\ell=2}^{\infty} {\omega^{\ast}(\ell)\over \ell \sqrt{{\rm log}\ \ell}} <\infty.
    	\end{eqnarray*}
    	Then the operator ${\mathcal M}_m$ is bounded on $L^p({\mathbb R^n})$.   	
    	\end{proposition}

To prove Proposition~\ref{prop00},
 Grafakos, Honz\'ik and Seeger \cite{GHS2006}
    considered    $N$ multipliers  $m_i,1\leq i\leq N $, satisfying
the condition  \eqref{e1.1} and asked for bounds:
\begin{eqnarray}\label{e1.2}
	\left\|\sup_{1\leq i\leq  N}  \big|{\mathcal F}^{-1}[m_i {\hat f}]\big|\,\right\|_{L^p(\RR^n)} \leq A(N)  \|f\|_{L^p(\RR^n)}.
\end{eqnarray}
They  obtained   an optimal bound  in $L^p({\mathbb R^n})$ with  $A(N) =O(\sqrt{{\rm log} (N+1)})$ provided that   $1\leq r<2$,   $p\in (r, \infty)$ and  the multipliers $m_i$, $i=1, \cdots, N,$  satisfy the condition
$$
\sup_{t>0} \left\|\phi m_i(t \cdot)\right\|_{W^2_{\alpha}({\mathbb R^n})}<\infty
\ \ \ {\rm for}\ \ \ i=1,2, \cdots, N
$$
and some $\alpha>n/r$.
The growth rate $A(N)=O(\sqrt{{\rm log} (N+1)}) $  in \eqref{e1.2} is known to be sharp due to the example  in
\cite{CGHS2005}, which  plays an essential role in the proof of
the result of Grafakos-Honz\'ik-Seeger \cite[Theorem 1.3]{GHS2006}, that is,
Proposition \ref{prop00} above.

 Later,  Choi \cite{C2015} extended the result of Grafakos-Honz\'ik-Seeger  \cite{GHS2006}     to the  multiplier operators on stratified groups.

\subsection{Spectral multipliers   }
The purpose of this paper is to extend  the result of Grafakos-Honz\'ik-Seeger  \cite[Theorem 1.3]{GHS2006}     to general spectral  multipliers on spaces of homogeneous type.
Let $(X,d,\mu)$ be a metric measure space with a metric $d$ and a measure $\mu$, which satisfies doubling measure, i.e.  there  exists a constant $C>0$ such that
\begin{eqnarray}
	\mu(B(x,2r))\leq C \mu(B(x, r)),\quad \forall\,r>0,\,x\in X. \label{eq2.1}
\end{eqnarray}
If this is the case, there exist  $C, n>0$ such that for all $\lambda\geq 1$ and $x\in X$
\begin{equation}
\mu(B(x, \lambda r))\leq C\lambda^n \mu(B(x,r)). \label{eq2.2}
\end{equation}
In the Euclidean space with Lebesgue measure, $n$ corresponds to
the dimension of the space.

 Suppose  that $L$ is a non-negative
self-adjoint  operator on~$L^2(X)$ and that
the semigroup $e^{-tL}$, generated by $-L$ on $L^2(X)$,  has the kernel  $T_t(x,y)$
which  satisfies
the following  Gaussian upper bound:
\begin{equation*}\label{ge}
	\tag{GE}
	\big|T_t(x,y)\big|\leq {C\over\mu(B(x,\sqrt{t}))} \exp\left(-{  d^2(x,y)\over {ct} } \right)
\end{equation*}
for all $t>0$,  and $x,y\in X,$   where $C$ and $ c$   are positive constants.
Such an operator admits a spectral
resolution  $E_L(\lambda)$  and for  any  bounded  Borel function $m\colon [0, \infty)
\to {\mathbb C}$, one can define the operator $m(L)$
\begin{equation*}
	m(L)=\int_0^{\infty}m(\lambda) \wrt E_L(\lambda).
\end{equation*}
By the spectral theorem, $m(L)$ is well defined and bounded on~$L^2(X)$. Spectral multiplier
theorems give sufficient conditions on~$m$ under which the operator~$m(L)$ extends to a
bounded operator on~$L^p(X)$ for some range of~$p$. This topic has attracted a lot of attention and has been studied extensively by many authors: for example,
for sub-Laplacian on nilpotent groups in \cite{Ch}, for sub-Laplacian on Lie groups of polynomial growth in \cite{A1994},
for Schr\"odinger operator on Euclidean space $\mathbb R^n$ in \cite{He}, for sub-Laplacian on
Heisenberg groups in \cite{MSt} and many others. For more information about the background of this topic,
the reader is referred to
\cite{A1994, BFP2016, COSY2016, Ch, DOS2002, He, MSt}  and the references therein.

We wish to point out \cite{DOS2002}, which is closely related to our paper. In \cite{DOS2002},
Duong, Ouhabaz and Sikora established the following  H\"ormander-type spectral  multiplier theorem for a
non-negative self-adjoint operator $L$   under  the assumption of the
kernel $T_t(x,y)$ of the analytic semigroup $e^{-tL}$ having
a Gaussian upper bound \eqref{ge}.

 \begin{proposition}\label{thmbydo}
 Let $L$ be a non-negative self-adjoint operator such that the corresponding
 heat kernels satisfy Gaussian bound  \eqref{ge}.
  Assume that for any $R>0$ and all Borel functions $m$  such that\, {\rm supp} $m\subseteq [0, R]$,
 \begin{eqnarray}\label{e1.4}
 	\int_X  \left|K_{m(\sqrt{L})}(x,y)\right|^2 d\mu(x) \leq {C\over \mu(B(y, R^{-1}))} \|  m(R\cdot)\|^2_{L^q}
 \end{eqnarray}
 for some $q\in[2,\infty]$. Let $s>n/ 2$.
Then for any bounded Borel function $m$
such that
\begin{eqnarray}\label{e1.44}
\sup_{t>0} \left\|\phi m( t\cdot)\right\|_{W^{q}_s}<\infty,
 \end{eqnarray} where $\phi\in C_c^\infty(\RR^+)$ is a fixed function, not identically zero,
the operator $m(L)$ is bounded on $L^p(d\mu)$ for all $1<p<\infty$, in addition,
\begin{align*}
  \|m(L)\|_{L^p(d\mu)\rightarrow L^p(d\mu)}\leq C_s \Big (|m(0)|+\sup_{t>0}\|\phi m( t\cdot)\|_{W^{q}_s}\Big).
\end{align*}
 \end{proposition}

\medskip

Note that Gaussian bound \eqref{ge} implies estimates (\ref{e1.4}) for $q=\infty$. This means that one can omit
condition (\ref{e1.4}) if the case $q=\infty$ is considered.
We call hypothesis  (\ref{e1.4})   the {\it Plancherel estimates} or {\it Plancherel
conditions}. For the standard Laplace operator on Euclidean spaces $\mathbb R^n$, this is equivalent to
$(1,2)$ Stein-Tomas restriction theorem (which is also the Plancherel estimate of the Fourier transform).

%

We consider the following two maximal functions formed by the dilations of a single multiplier associated with operator $L$:
 \begin{eqnarray*}
 M^{\rm dyad}_{m, L}f(x)=\sup_{k\in {\mathbb Z}} \left |   m(2^kL) f(x) \right|
\end{eqnarray*}
and
\begin{eqnarray*}
 M_{m, L}f(x)=\sup_{t>0}   \left|   m(tL) f(x) \right|.
 \end{eqnarray*}


From the example in \cite{CGHS2005}, Mikhlin-H\"ormander type assumption
in \eqref{e1.44} or \eqref{e1.1} alone is not sufficient to guarantee the $L^p$-boundedness of  $M_{m, L}$.
Additional decay assumptions should be involved. Let $\phi\in C_c^\infty(\RR)$  supported in $\big \{\xi:1/2\leq  |\xi|\leq 2\big\}$ which is not zero. In Section \ref{sec:3}, we prove the following result of the boundedness of $M_{m, L}$.
%
%

\begin{theorem}\label{Theorem2}
 Suppose that   $L$ is a self-adjoint operator that satisfies   Gaussian estimate  \eqref{ge} and Plancherel condition \eqref{e1.4} for some $q\in[2,\infty]$. Suppose   $m$ is a bounded Borel function satisfying
  \begin{align}\label{eq1.2.3}
  \Big (\sum_{k\in\ZZ} \|\phi m(2^k\cdot)\|_{W^2_{s}}^2\Big )^{1/2}+|m(0)|<\infty
 \end{align}
for some $s>n/2-1/q+1$, then $M_{m,L}$ is bounded on $L^p(d\mu)$ for all $1<p< 2$.
 If  $m$ satisfies \eqref{eq1.2.3} for some $s>n/2-1/q+1/2$, then $M_{m,L}$ is bounded on $L^p(d\mu)$ for all $2\leq p<\infty$.
\end{theorem}





Let $\{\nu(k)\}_{k=0}^\infty$ be a positive increasing and unbounded sequence. Christ,   Grafakos,  Honz\'ik and Seeger \cite{CGHS2005} constructed   a multiplier function $m$ satisfying
\begin{align}\label{eq1.8.8}
\sup_\xi \left |\partial_\xi^\alpha (\phi(\xi)m(2^k\xi)) \right |\leq C_\alpha\frac{\nu(|k|)}{\sqrt{1+\log|k|}},\quad k\in \mathbb Z
\end{align}
with $C_\alpha$ for all $\alpha\in \mathbb N^+$ so that $M_{m,L}$ is unbounded on $L^p(\RR^n)$ for all $1<p<\infty$.

In terms of the positive result of Theorem~\ref{Theorem2}, we note that there is a significant gap between the
the condition in \eqref{eq1.2.3} and the weak decay \eqref{eq1.8.8}. Assuming $\|\phi m(2^k\cdot)\|_{W^2_{s}}=O(|k|^{-\alpha})$, then \eqref{eq1.2.3} yields $L^p$ boundedness for all $1<p<\infty$ when $\alpha>1/2$.
We shall see that this result remains in fact valid under the weaker assumption
$$
\left\|\phi m(2^k\cdot)\right \|_{W^2_{s}}\leq C \big(\sqrt{1+\log|k|}\big)^{-1-\epsilon}.
$$
In what follows we shall mainly aim for minimal decay but will also try to formulate
reasonable smoothness assumptions.

In order to do this, more assumptions on the operator $L$ are needed.
 Besides Gaussian upper bound~\eqref{ge}, in some applications, the operator $L$ still satisfies the following conditions. There exists a positive  function $h\in L^1_{loc}(X)$ with $h:X\rightarrow (0,\infty)$ so that

\begin{itemize}
\item[(H-1).] For any $t>0$
\begin{align*}
  T_t(h)(x)=h(x),\ \ {\rm{a.e. }} \ x\in X.
\end{align*}
\item[(H-2).] There exists constant $C>0$ such that for any $x,y\in X$ and $t>0$,
\begin{align}\label{LUGh}
\tag{$\mathrm{ULG_h}$}
\frac{C^{-1}}{\mu_{h^2}(x,\sqrt{t})}\exp\left(-\frac{d(x,y)^2}{c_1t}\right)\leq \frac{T_t(x,y)}{h(x)h(y)}\leq \frac{C}{\mu_{h^2}(x,\sqrt{t})}\exp\left(-\frac{d(x,y)^2}{c_2t}\right),
\end{align}
\end{itemize}
\noindent where $\mu_{h^2}$ is the measure with the density $h^2d\mu$.

Based on these conditions, we would use the Doob transform to handle the cancellation properties of martingale difference and Littlewood-Paley operators, which is a necessary observation in \cite{GHS2006}.
We want to point out that the condition~\rm{(H-1}) is automatically satisfied. Indeed, if there exists a function $h$ such that the condition~(H-2) hold, then conditions~\rm{(H-1)} and \rm{(H-2)} are  true for $\widetilde{h}=\varphi h$, where $0<C^{-1}\leq\varphi\leq C<\infty$ (see \cite[Propostion~2.3]{PSY2022}). The function $h$ in assumption \rm{(H-1)} may be unbounded, such as  Schr\"odinger operators with inverse-square potential in $\RR^n,n\geq3$ (see Section \ref{sec:7}).

\medskip
Our main result is concerned with the  maximal functions  $ M^{\rm dyad}_{m, L}$ and $M_{m, L}$.
To continue, we first recall the definition of a  rearrangement of a sequence $\omega$, for $t\geq0$,
\begin{align*}
  \omega^*(t)=\sup\big\{\lambda>0: \mathrm{Car} \left(k:|\omega(k)>\lambda| \right)>t\big\},
\end{align*}
where $\mathrm{Car}(E)$ is the cardinality of a given set $E$.
We have the following results.

\begin{theorem}\label{Theorem3}
 Suppose that   $L$ is a self-adjoint operator that satisfies  Gaussian estimate  \eqref{ge} and Plancherel condition \eqref{e1.4} for some $q\in[2,\infty]$. Assume there exists an $L$-harmonic $h$ such that conditions  \rm{(H-1)} and \eqref{LUGh} hold.
Suppose  $1\leq r<2$ and  $r<p<\infty$. For some $s>0$, the bounded Borel function $m$ satisfies
\begin{align}\label{eq1.4.2}
   \left \|\phi m(2^k\cdot) \right \|_{W^q_{s}}\leq \omega(k).
\end{align}
Assume that the non-increasing rearrangement function $\omega^{*}$ satisfies
\begin{align}\label{eq1.4.1}
\omega^{*}(0)+\sum_{\ell\geq2}\frac{\omega^{*}(\ell)}{\ell\sqrt{\log\ell}}<\infty.
\end{align}
 Let $h^{p-2}\in A_{p/r}(h^2d\mu)$, then the following statements hold:

(i) Assume $s>n/r$, then $M_{m,L}^{\mathrm{dyad}}$ is bounded on $L^p(d\mu)$.

(ii) Assume $s>n/r+1/p$, then $M_{m,L}$ is bounded on $L^p(d\mu)$.

\end{theorem}

%
%

Similar as the classical case, by Calder\'on-Zygmund theory, we can show that if $M_{m,L}$ is a priori bounded on $L^2(X)$ and
if $\sup_{t>0}\|\phi\, m( t\cdot)\|_{W^{2}_s}<\infty$ holds for some $s>n/2+1/q'$, then $M_{m,L}$ is weak-type $(1,1)$ and thus bounded on $L^p$ for all  $1<p\leq 2$. See Proposition~\ref{H1toL1} in Section \ref{sec:6} for details. Then by interpolation, we have the following result which states that $h^{p-2}$ belonging to some weight class is not necessary for $1<p\leq 2$.

\begin{theorem}\label{Theorem8}
   Suppose that   $L$ is a self-adjoint operator that satisfies  Gaussian estimate  \eqref{ge} and Plancherel condition \eqref{e1.4} for some $q\in[2,\infty]$. Assume there exists an $L$-harmonic $h$ such that conditions  \rm{(H-1)} and \eqref{LUGh} hold. Let $1<p\leq2$ and
   $s>n/2+1/q'$.
   Suppose    the bounded Borel function $m$ satisfies
\begin{align}\label{tm7c1}
\left\|\phi m(2^k\cdot)\right\|_{W^{q}_{s}}\leq \omega(k).
\end{align}
Assume that the non-increasing rearrangement function $\omega^{*}$ satisfies
\begin{align*}
\omega^{*}(0)+\sum_{\ell\geq2}\frac{\omega^{*}(\ell)}{\ell\sqrt{\log\ell}}<\infty.
\end{align*}
Then $M_{m,L}$ is bounded on $L^p(d\mu)$.

\end{theorem}

%

\medskip
The proof of Theorems~\ref{Theorem3} and  ~\ref{Theorem8} is based on  some estimates for the  maximal function $\sup_{1\leq i\leq N}|m_i(L)(f)|$ for given multipliers $m_i,i=1,\cdots, N,$ with uniform estimates, which   can be stated in the following way.

\begin{theorem}\label{Theorem1}
 Suppose that   $L$ is a self-adjoint operator that satisfies   Gaussian estimate  \eqref{ge} and Plancherel condition \eqref{e1.4} for some $q\in[2,\infty]$. Assume there exists an $L$-harmonic $h$ such that conditions  \rm{(H-1)} and \eqref{LUGh} hold.
Let $1\leq r<2$ and $N$ functions  $m_i$, $i=1, \cdots, N ,$ satisfy
  $$
 \sup_{t>0}  \big\|\phi m_i(t\cdot)\big\|_{W^q_{s}}+|m_i(0)| \leq B<\infty,\ \ \mathrm{for}\ i=1,2,\cdots, N
  $$
  and  some $s>n/r$.
 Then $\sup_{1\leq i\leq N}|m_i(L)(f)|$ is bounded on $L^p(d\mu)$ for all $p$ and $h$ satisfying $p\in(r,\infty)$ and $h^{p-2}\in A_{p/r}(h^2d\mu)$. In addition,
  \begin{align}\label{a1}
    \Big\|\sup_{1\leq i\leq N} \left|m_i(L)(f)\right | \Big \|_{L^p(d\mu)}\leq C_{p,r}B\sqrt{\log(1+N)}\|f\|_{L^p(d\mu)}.
  \end{align}
\end{theorem}

\smallskip

%
%
%

%
%
%
%


We mention that the proof of Theorem~\ref{Theorem1}
is inspired by the result of
Grafakos-Honz\'{i}k-Segger \cite{GHS2006}, which was shown  by making use of the ${\rm exp}(L^2)$ estimate by Chang-Wilson-Wolff \cite{CWW1985} for functions with bounded Littlewood-Paley square function.
However, we can not use the technique of proof of Chang-Wilson-Wolff (\cite[Theorem 3.1]{CWW1985})  in this step. The main reason is that with application to the
Schr\"odinger operator in mind, we do not  assume that the semigroup
$\{e^{-tL}\}_{t>0}$  satisfies the  H\"older regularity and the  conservation property  such that
$$
e^{-tL}1 = 1.
$$
To   overcome the obstacles,  our approach is to apply the Doob transform
to get a new semigroup  $\mathcal{T}_t$, and its generator $\mathcal{L}$ is also
self-adjoint operator satisfying  the  H\"older regularity of the semigroup   kernel  and the  conservation property.

This paper is organized as follows. In Section \ref{sec:2}, we recall some preliminary known results on the   Littlewood-Paley function associated with operator $L$, Muckenhoupt weights and  Doob transform. In Section \ref{sec:3}, we prove Theorem~\ref{Theorem2}. In Section \ref{sec:4}, we will prove a weighted good-$\lambda$ estimate of the square function for the dyadic martingale and give some estimates of  the Littlewood-Paley operators, which will play an important role in the proof of Theorem~\ref{Theorem1}.    The proof of Theorem \ref{Theorem1} will be given in   Section \ref{sec:5}. The proof of Theorems~\ref{Theorem3} and \ref{Theorem8} will be given in Section \ref{sec:6}.  In Section \ref{sec:7}, we will give some examples  to obtain some new applications of Theorems~\ref{Theorem3} and \ref{Theorem8}.

\subsection{Notations}
Throughout this paper,
unless we mention the contrary, $(X,d,\mu)$ is a metric measure  space, that is, $\mu$
is a Borel measure with respect to the topology defined by the metric $d$.
Given a Borel measure $\mu$ and $1\leq p\leq \infty$, we denote by $L^p(d\mu)$ the $L^p$ space equipped with the measure $\mu$.
We write $\mathcal{S}$ to mean classes of the Schwartz functions.
We denote by
$B(x,r)=\{y\in X,\, {d}(x,y)< r\}$  the open ball
with centre $x\in X$ and radius $r>0$. We   often just use $B$ instead of $B(x, r)$.
Given $\lambda>0$, we write $\lambda B$ for the $\lambda$-dilated ball
which is the ball with the same centre as $B$ and radius $\lambda r$.
Given a function $f\in\mathcal{S}$, we will use $\mathcal{F}f$  or  $\hat{f}$ to represent the Fourier transform of $f$; use $\mathcal{F}^{-1}f$  or  $\check{f}$ to represent the Fourier inverse transform of $f$.

\medskip
\section{Preliminaries}\label{sec:2}
\setcounter{equation}{0}

In this section, we will recall some preliminary known results, which will be used to prove Theorems~\ref{Theorem1} and \ref{Theorem3}.

Let $Q^\alpha_k$ denote the cube with side-length $\delta^k$, where $0<\delta<1$.
\begin{lemma}\label{le2.1}
	Let $(X, d, \mu) $ be a space of homogeneous type.
	  There exists a system of dyadic cubes $\{Q^\alpha_k\}_{k,\alpha}$ for constant $\delta$. The dyadic cubes $\{Q^\alpha_k\}_{k,\alpha}$ satisfy the following properties:
	\begin{enumerate}
		\item  for each $k\in\ZZ$, $X=\cup_{\alpha}Q^\alpha_k$ and $\{Q^\alpha_k\}_{k,\alpha}$ are disjoint for different $\alpha$.
		\item if $\ell\leq k$, for any $\alpha,\beta$, either $Q^\alpha_k\subseteq Q^\beta_\ell$ or $Q^\alpha_k\cap Q^\beta_\ell=\varnothing$.
		\item For each $Q^\alpha_\ell$ and $k\leq \ell$, there  exists a unique  $Q^\beta_k$ so that $Q^\alpha_\ell\subseteq Q^\beta_k$.
		\item there exists a positive integer $N(\delta,n)$ so that for each $(\ell,\alpha)$,
		$$
		1\leq \#\left\{\beta:\    Q^\beta_{\ell+1}\subseteq Q^\alpha_\ell\right\}\leq N(\delta,n)\ \  \mathrm{and}\ \
		Q^\alpha_\ell=\bigcup_{\beta:Q^\beta_{\ell+1}\subseteq Q^\alpha_\ell}Q^\beta_{\ell+1}.
		$$
		\item for each $(\ell,\alpha)$, there exists a $x^\alpha_\ell\in X$ so that
		$$B(x^\alpha_\ell,\delta^\ell)\subseteq Q^\alpha_\ell\subseteq B(x^\alpha_\ell,C\delta^\ell).$$
		\item for any $0<t\leq1$ and $(\alpha,k)$, there exists  constants $\rho>0$ and $C>1$ so that
		\begin{align*}
			\mu\left(\left\{x\in Q_k^\alpha:\   d(x, (Q_k^\alpha)^c)\leq t \delta^k\right\}\right)\leq Ct^{\,\rho} \mu(Q_k^\alpha).
		\end{align*}
	\end{enumerate}
\end{lemma}
\begin{proof}
	The proof of properties~(1)-(5) can be  referred to  \cite[Theorem~2.2]{HK2012}.
The proof of (6) can be referred to \cite[Chapter~VI, Theorem~14]{C1990}.
	
\end{proof}

\subsection{The Littlewood-Paley function}\label{subsec:2.1}
Assume that $\psi\in C_c^\infty(\RR)$ is an even function with $\int \psi(t) dt=0$, $\psi\neq 0$ and supp\,$\psi\in[-1,1]$. Let $\Psi$ denote  the Fourier transform of $\psi$.  Fix $\lambda>1$. We consider the following Littlewood-Paley $G$-function associated with $L$,
$$G_L^\lambda(f)(x)
:=\bigg(\int_{0}^{\infty}\int_X\left(\frac{t}{t+d(x,y)}\right)^{\lambda n}  \left|\Psi(t\sqrt{L})f(y)\right|^2
\frac{d\mu(y)}{\mu(B(x,t))}\frac{dt}{t}\bigg)^{\frac12}.
$$
For every   $j\in {\mathbb Z}$, we define
\begin{align}\label{ujl}
	U_{j,L}^\lambda(f)(x):=\bigg(\int_{2^{-j-1}}^{2^{-j}}\int_X\left(\frac{t}{t+d(x,y)}\right)^{\lambda n}  \left|\Psi(t\sqrt{L})f(y)\right|^2\frac{d\mu(y)}{\mu(B(x,t))}\frac{dt}{t} \bigg)^{\frac12},
\end{align}
and so
$$
\bigg(\sum_{j\in\ZZ} |U_{j,L}^\lambda(f)(x)|^2\bigg)^{1/2}=G_L^\lambda(f)(x).
$$
From \cite[Proposition~3.3]{GY2012}, we have
\begin{lemma}
 Let $L$ be a non-negative self-adjoint operator such that the corresponding
heat kernels satisfy Gaussian bound  \eqref{ge}.
For any $f\in L^p(d\mu)$,  there exists a constant $C$ such that
	\begin{align}\label{GLI}
		\|G_{L}^\lambda(f)\|_{L^p(d\mu)}\leq C\|f\|_{L^p(d\mu)}.
	\end{align}
	provided either $1<p<2$ and $\lambda>2/p$, or $2\leq  p<\infty$ and $\lambda>1$.
\end{lemma}
We show the following pointwise control of multipliers by the Littlewood-Paley $G$-function.
\begin{lemma}\label{Lemma3}
 Let $L$ be a non-negative self-adjoint operator such that the corresponding
heat kernels satisfy Gaussian bound  \eqref{ge}
and Plancherel estimate \eqref{e1.4} for some $q\in[2,\infty]$.  Let $s>n/2$, $k\in\NN$ and $k>s/2+1$. Then for any bounded Borel function $F$ which satisfies  $\sup_{t>0}\|\phi F(t\cdot)\|_{W^q_s}<\infty$, and for any $n/2<s_0<s$,
 \begin{align*}
&\left|(t^2L)^k\exp(-t^2L)\Psi(t\sqrt{L})F(L)f(x)\right|\\
&\leq C\Big(\sup_{t>0}\|\phi F(t\cdot)\|_{W^q_s}+|F(0)|\Big)\left(\int_X\left(\frac{t}{t+d(x,y)}\right)^{2s_0}\left|\Psi(t\sqrt{L})f(y)\right|^2
\frac{d\mu(y)}{\mu(B(x,t))}\right)^{\frac12}.
\end{align*}

\end{lemma}

\begin{proof} Write
 \begin{align} \label{L3_1}
 F(L)=F(L)-F(0)I+F(0)I.
 \end{align}
 Let $\widetilde{F}(t)=F(t)-F(0)$. For $k>s/2+1$ and $s>s_0>n/2$,
 \begin{align}\label{L3_2}
   &\left|(t^2L)^k\exp(-t^2L)\Psi(t\sqrt{L})\widetilde{F}(L)f(x)\right|\nonumber\\
   &=\bigg|\int_{X}K_{(t^2L)^k\exp(-t^2L)\widetilde{F}(L)}(x,y)\Psi(t\sqrt{L})f(y)d\mu(y)\bigg|\nonumber\\
   &\leq \left(\int_X \left(\frac{t}{t+d(x,y)}\right)^{-2s_0} \left|K_{(t^2L)^k\exp(-t^2L)\widetilde{F}(L)}(x,y)\right|^2d\mu(y)\right)^{1/2} \nonumber\\
   &\ \ \ \ \times\left(\int_X\left(\frac{t}{t+d(x,y)}\right)^{2s_0} \left |\Psi(t\sqrt{L})(f)(y)\right|^2
d\mu(y)\right)^{\frac12}\nonumber\\
&\leq C\sup_{t>0}\|\phi F(t\cdot)\|_{W^q_s}\left(\int_X\left(\frac{t}{t+d(x,y)}\right)^{2s_0} \left|\Psi(t\sqrt{L})f(y)\right|^2
\frac{d\mu(y)}{\mu(B(x,t))}\right)^{\frac12}.
\end{align}
where the last inequality follows from  \cite[Lemma~4.5 and Lemma~4.6]{GY2012}  that
$$\int_X \left(1+\frac{d(x,y)}{t}\right)^{2s_0} \left |K_{(t^2L)^k\exp(-t^2L)F(L)}(x,y)\right|^2d\mu(y)\leq \frac{C}{\mu(B(x,t))}\sup_{t>0}\|\phi F(t\cdot)\|_{W^q_s}^2.
$$
Also for $s_0>n/2$ and $k\in\NN$, we have
  \begin{align}\label{L3_3}
   &\left|(t^2L)^k\exp(-t^2L)\Psi(t\sqrt{L})f(x)\right|\nonumber\\
   &\leq \left(\int_X \left(\frac{t}{t+d(x,y)}\right)^{-2s_0} \left|K_{(t^2L)^k\exp(-t^2L)}(x,y)\right|^2d\mu(y)\right)^{\frac12}\nonumber\\
   &\ \ \ \ \times\left(\int_X\left(\frac{t}{t+d(x,y)}\right)^{2s_0} \left|\Psi(t\sqrt{L})f(y)\right|^2
d\mu(y)\right)^{\frac12}\nonumber\\
&\leq C\left(\int_X\left(\frac{t}{t+d(x,y)}\right)^{2s_0} \left|\Psi(t\sqrt{L})f(y)\right|^2
\frac{d\mu(y)}{\mu(B(x,t))}\right)^{\frac12}.
\end{align}
 Combining \eqref{L3_1}-\eqref{L3_3}, we  obtain
 \begin{align*}
&  \left |(t^2L)^k\exp(-t^2L)\Psi(t\sqrt{L})F(L)f(x)\right|\\
&\leq C\Big (\sup_{t>0}\|\phi F(t\cdot)\|_{W^q_s}+|F(0)|\Big)\left(\int_X\left(\frac{t}{t+d(x,y)}\right)^{2s_0} \left|\Psi(t\sqrt{L})f(y)\right|^2
\frac{d\mu(y)}{\mu(B(x,t))}\right)^{\frac12}.
\end{align*}
This completes the proof.
\end{proof}

\subsection{Muckenhoupt weights}\label{subsec:2.2}


Let $1\leq p <\infty$ and $\mu$ is a doubling measure. Let us recall the definition of $ A_p(d\mu)$. For $p=1$, $\omega\in A_p(d\mu)$  means that for any ball $B\subseteq X$,
$$\frac{1}{\mu(B)}\int_B\omega(y)d\mu(y) \leq C\omega(x), \mathrm{\ \ for\ \ } \mathrm{a.e.} \ x\in B.$$
For $p>1$, $\omega\in A_p(d\mu)$  means that
$$[\omega]_{A_p(d\mu)}:=\sup_{B:balls}\left(\frac{1}{\mu(B)}\int_{B}\omega(y)d\mu(y)\right)^{\frac1p}
\left(\frac{1}{\mu(B)}\int_{B}\omega(y)^{-\frac{1}{p-1}}d\mu(y) \right)^{\frac1{p'}}<\infty.$$
We set $A_\infty(d\mu)=\bigcup_{p\geq1}A_p(d\mu).$


\begin{lemma}\label{RHL}
Let $1\leq p\leq \infty$. If $\omega\in A_p(d\mu)$, then the following statements hold for classes $A_p(d\mu)\mathrm{:}$

\rm{(i)} $A_{p}(d\mu)\subseteq  A_q(d\mu)$  for $1\leq p\leq q \leq \infty $.

\rm{(ii)} There exists a constant $\gamma>0$ such that for any ball $B$,
\begin{align*}
\left(\frac{1}{\mu(B)}\int_B\omega(y)^{1+\gamma}d\mu\right)^{\frac{1}{1+\gamma}}\leq \frac{C}{\mu(B)}\int_B\omega(y)d\mu.
\end{align*}


\rm{(iii)} For any ball $B\subseteq X$, let $A$ be the measurable subset of $B$, there exists a $\delta\in(0,1)$ such that
\begin{align*}
  \frac{\omega(A)}{\omega(B)}\leq C \left(\frac{\mu(A)}{\mu(B)}\right)^\delta.
\end{align*}

\rm{(iv)} For $1<p<\infty$ and any ball $B\subseteq X$, there exists constants $C_1,C_2>0$ such that
$$C_1\mu(B)\leq \omega(B)^{\frac1{p}}\left(\omega^{-\frac1{p-1}}(B) \right)^{\frac1{p'}}\leq C_2\mu(B).$$
\end{lemma}

\begin{proof}
  For the proof, we refer the reader to \cite{AM2007,G2014,ST1989}.
\end{proof}

\subsection{The Doob transform} \label{subsec:2.3}
In this subsection we describe one of the important tools for this paper, i.e. the Doob transform (or $h$-transform),
see  \cite{Gyrya-Saloff-Coste}. Assume that an operator $L$ related to a metric measure space
$(X,d,\mu)$ and
   $L$ is a self-adjoint operator that satisfies conditions \rm{(H-1)} and \eqref{LUGh}.

\begin{lemma}\label{Lemma9}
Assume that the kernel $T_t(x,y)$ satisfies  condition~\eqref{LUGh} for some function $h$. Then the measure $d\upsilon=h^2d\mu$ satisfies the doubling property.
\end{lemma}
\begin{proof}
  For the proof, we refer the reader to \cite[Proposition~2.4]{PS2024}.
\end{proof}


With $L$-harmonic function $h$, we define a new measure by
$
d\upsilon(x)=h(x)^2d\mu(x),
$
the space $(X, d, \upsilon)$ satisfies the doubling condition.
Define
\begin{align*}
\mathcal{T}_t(x,y)=\frac{T_t(x,y)}{h(x)h(y)}.
\end{align*}
 The  Doob transform is a multiplication operator
$$
f\mapsto h^{-1}f.
$$
Observe that
$$
\|f\|_{L^2(d\mu)}=\|h^{-1}f\|_{L^2(d\upsilon)},
$$
so the Doob transform
is an isometric isomorphism between these two $L^2$ spaces. Moreover, a calculation shows
that $\mathcal{T}_t$ is a semigroup and its generator $\mathcal{L}$, which plays the image of $L$ under Doob transform,  is also
self-adjoint.

\begin{lemma}\label{Dh}
	Suppose that   $L$ is a self-adjoint operator that satisfies conditions \rm{(H-1)} and \eqref{LUGh}. Then the operator $\mathcal{L}$ satisfies the following properties:
	\begin{enumerate}
	\item [(i)] $\mathcal{T}_t(x,y)$ satisfies the  lower and upper Gaussian  bounds, i.e.
$$\frac{C^{-1}}{\upsilon(B(x,\sqrt{t}))}\exp\left(-\frac{d(x,y)^2}{c_1t}\right)\leq |\mathcal{T}_t(x,y)|\leq  \frac{C}{\upsilon(B(x,\sqrt{t}))}\exp\left(-\frac{d(x,y)^2}{c_2t}\right);$$

	\item [(ii)] $\mathcal{L}$ has  the conservation property. That is,
$$\int_X \mathcal{T}_t(x,y)d\upsilon(y) =1,\ \ \mathrm{for\ all\ }\ \ t>0$$

	\item [(iii)] There exist  constants $C,c$ and $\gamma>0$ such that for $d(y,z)\leq \sqrt{t}$,
\begin{align}\label{Tf1}
|\mathcal{T}_t(x,y)-\mathcal{T}_t(x,z)|\leq \frac{C}{\upsilon(B(x,\sqrt{t}))}\left(\frac{d(y,z)}{\sqrt{t}}\right)^{\gamma}\exp\left(-\frac{d(x,y)^2}{ct}\right).
\end{align}
\end{enumerate}
\end{lemma}
\begin{proof}
the result of \rm{(i)} and  \rm{(ii)} can follow from conditions~\rm{(H-1)}, \eqref{LUGh}, $d\upsilon=h^2d\mu$ and  $\mathcal{T}_t(x,y)=T_t(x,y)/(h(x)h(y))$.
  For the proof of  \rm{(iii)}, we refer the reader to {\color{red}\cite[Theorem~5]{DP2017}}.
\end{proof}
\medskip

Let   $\mathcal{K}_t(x,y)$  denote the kernel of $t\mathcal{L}\exp(-t\mathcal{L})$. From \cite[Theorem~4]{S2004}, we have
$$\mathcal{K}_t(x,y)\leq \frac{C}{\upsilon(B(x,\sqrt{t}))}\exp\left(-\frac{d(x,y)^2}{c_3t}\right).$$
Observe that
$$\mathcal{K}_t(x,y)=2\int_X \mathcal{K}_{t/2}(x,z)\mathcal{T}_{t/2}(z,y)d\upsilon(z).
$$
This, in combination with \eqref{Tf1},  gives  that for $d(y,z)\leq \sqrt{t}$, there exists a $\gamma>0$ such that
\begin{align}\label{Tf2}
|\mathcal{K}_t(x,y)-\mathcal{K}_t(x,z)|\leq \frac{C}{\upsilon(B(x,\sqrt{t}))}\left(\frac{d(y,z)}{\sqrt{t}}\right)^{\gamma}\exp\left(-\frac{d(x,y)^2}{c_4t}\right).
\end{align}

\medskip

%
%
%
%
%
%

\smallskip

Recall that   a new measure is defined by
$
d\upsilon(x)=h(x)^2d\mu(x),
$
and
the space $(X, d, \upsilon)$ satisfies the doubling property.
By Lemma~\ref{le2.1},  we obtain a system of dyadic cubes   $\{Q^\alpha_k\}_{k,\alpha}$.   Denote by $\mathcal{D}_k$ the family of dyadic cubes of length $\delta^k$.
Let $Q_k(x)\in \mathcal{D}_k$ is  a  dyadic cube with length $\delta^{k}$ and   $x\in Q_k(x)$.
On space  $(X,d,\upsilon)$, let us denote by $\mathbb{E}^\upsilon_k$ the expectation operator
\begin{eqnarray}\label{e2.9}
\mathbb{E}_k^{\upsilon}(g)(x):=\sum_{Q\in \mathcal{D}_k}\chi_{Q}(x)\frac{1}{\upsilon(Q)}\int_{Q}g(y)d\upsilon(y)
:=\frac1{\upsilon(Q_k(x))}\int_{Q_{k}(x)}g(y)d\upsilon(y),
\end{eqnarray}
denote by $\mathbb{D}^\upsilon_k$ the martingale differences,
 \begin{eqnarray}\label{e2.10}
	\mathbb{D}^\upsilon_k(g)(x)&:=&\mathbb{E}^\upsilon_{k+1}(g)(x)-\mathbb{E}^\upsilon_k(g)(x)\nonumber\\
	&=:&\frac{1}{\upsilon(Q_{k+1}(x))}
	\int_{Q_{k+1}(x)}g(y)d\upsilon(y)-\frac{1}{\upsilon(Q_k(x))}\int_{Q_{k}(x)}g(y)d\upsilon(y)
	,
\end{eqnarray}
and denote by $\mathbb{S}^\upsilon$  the square function for dyadic martingale in $(X,d,\upsilon)$,
\begin{eqnarray}\label{e2.11}
	\mathbb{S}^\upsilon(g)(x):=\bigg(\sum_{k\in \ZZ} \left\|\mathbb{D}^\upsilon_k(g)\right\|^2_{L^\infty(Q_k(x),d\upsilon)}\bigg)^{1/2}.
\end{eqnarray}

In \cite{CWW1985}
Chang, Wilson and Wolff  proved the exponential-square  integrability and the good-$\lambda$ inequality of square function for dyadic martingale on $\RR^n$. Later,  Wu and Yan \cite{WY2018} extended the result to  the space of  homogeneous type, and it can be stated  in the following way.
\begin{lemma}\label{Lemma5}
	Assume $(X,d,\upsilon)$ satisfies the doubling property.  For any $f\in L^1_{loc}(X,d\upsilon)$,  there is a  constant $c_n$ so that for some $0<\varepsilon<1/2$,
	\begin{eqnarray}\label{ee1}
	 &&\upsilon\left(\left\{x\in X:   \   \sup_{k\in\ZZ}\left|\mathbb{E}^\upsilon_k(f-f_X)(x)\right|>2\lambda,  \   \    \mathbb{S}^\upsilon(f)(x)<\varepsilon\lambda\right\}\right)\nonumber\\
	&&	   \leq C\exp\left(-\frac{c_n}{\varepsilon^2}\right)\upsilon\left(\left\{x\in X:\   \sup_{k\in\ZZ}  \left|\mathbb{E}^\upsilon_k(f-f_X)(x)\right|>\lambda\right\}\right),\ \ \forall \lambda>0,
	\end{eqnarray}
where $f_X=\upsilon(X)^{-1}\int_Xfd\upsilon$ when $0<\upsilon(X)<\infty$ and  $f_X=0$  when $\upsilon(X)=\infty$.
\end{lemma}

In the weighted situation, the  good-$\lambda$ inequality of the square function for dyadic martingale  also holds in the following way.
\begin{lemma}\label{Lemma8}
Let $f\in L^{p_0}(d\upsilon)$ for some $p_0\in[1,\infty]$.
Assume $(X,d,\upsilon)$ satisfies the doubling property. Let $1\leq p\leq\infty$ and $\omega\in A_p(d\upsilon)$. Then there are  constants $c_n,C>0$ so that for some $0<\varepsilon<1/2$,
\begin{align}\label{ee8}
  &\omega\left(\left\{x\in X:\   \sup_{k\in\ZZ}\left|\mathbb{E}^\upsilon_k(f-f_X)(x)\right|>2\lambda, \   \     \mathbb{S}^\upsilon(f)(x)<\varepsilon\lambda\right\}\right)\nonumber\\
 & \leq C\exp\left(-\frac{c_n}{\varepsilon^2}\right)\omega\left(\left\{x\in X:\    \sup_{k\in\ZZ} \left|\mathbb{E}^\upsilon_k(f-f_X)(x)\right|>\lambda\right\}\right)\ \ \forall \lambda>0,
\end{align}
where $f_X=\upsilon(X)^{-1}\int_Xfd\upsilon$ when $0<\upsilon(X)<\infty$ and  $f_X=0$  when $\upsilon(X)=\infty$.
\end{lemma}

\begin{proof}
Let us discuss the lemma into two cases: $\upsilon(X)=\infty$ and $0<\upsilon(X)<\infty$.

\noindent{\bf{Case~1: $\upsilon(X)=\infty$. }}

$f\in L^{p_0}(d\upsilon)$ for some $p_0\in[1,\infty]$. Then from the weak boundedness of Hardy-Littlewood maximal function, we have
  \begin{align*}
  \upsilon\left(\left\{x\in X:\   \sup_{k\in\ZZ}\left|\mathbb{E}^\upsilon_k(f)(x)\right|>\lambda\right\}\right)
  &\leq C\lambda^{-p_0}\left\|f\right\|_{L^{p_0}(d\upsilon)}^{p_0}<\infty.
  \end{align*}
If $\upsilon\left(\left\{x\in X:\  \sup_{k\in\ZZ}\left|\mathbb{E}^\upsilon_k(f)(x)\right|>\lambda\right\}\right)=0$, then
   $$
   \upsilon\left(\left\{x\in Q:\     \sup_{k\in\ZZ}\left|\mathbb{E}^\upsilon_k(f)(x)\right|>\lambda\right\}\right)=0,  \quad\forall \,Q\in \mathcal{D}.
   $$
  From the H\"older inequality and the  reverse H\"older inequality, there exists a $\gamma>0$ such that
   \begin{align*}
    & \omega\left(\left\{x\in Q:\   \sup_{k\in\ZZ} \left|\mathbb{E}^\upsilon_k(f)(x)\right|>\lambda\right\}\right)\\
     &\leq \left(\int_Q\omega(y)^{1+\gamma}d\upsilon(y)\right)^{\frac{1}{1+\gamma}}\upsilon\left(\left\{x\in Q:\   \sup_{k\in\ZZ}  \left|\mathbb{E}^\upsilon_k(f)(x)\right|>\lambda\right\}\right)^{\frac{\gamma}{\gamma+1}}\\
     &\leq C\omega(Q)\bigg (\upsilon\left(\left\{x\in Q:\   \sup_{k\in\ZZ} \left|\mathbb{E}^\upsilon_k(f)(x)\right|>\lambda  \right\}\right)/\upsilon(Q)  \bigg)^{\frac{\gamma}{\gamma+1}}=0.
   \end{align*}
 Therefore, for any dyadic cube $Q$,
   $$
   \omega\left(\left\{x\in Q:\  \sup_{k\in\ZZ}\left|\mathbb{E}^\upsilon_k(f)(x)\right|>\lambda\right\}\right)=0.
   $$
Then  estimate~\eqref{ee8}  obviously holds.

Let us consider
   $$
   0<\upsilon\left(\left\{x\in X:\   \sup_{k\in\ZZ}  \left|\mathbb{E}^\upsilon_k(f)(x)\right|>\lambda\right\}\right)<\infty.
   $$
   Let $\{Q_{\lambda,i}\}$ be the maximal dyadic cubes such that $\sup_{Q_{\lambda,i}}|f_{Q_{\lambda,i}}|>\lambda$. We know that these $\{Q_{\lambda,i}\}$ are disjoint and satisfy
   $$
   \left\{x\in X:\ \sup_{k\in\ZZ} \left|\mathbb{E}^\upsilon_k(f)(x)\right|>\lambda\right\}=\bigcup_{i\in\ZZ} Q_{\lambda,i}.
   $$
   Then to prove estimate~\eqref{ee8}, it just need to prove that
 \begin{align}\label{ee9}
  &\omega\left(\left\{x\in Q_{\lambda,i}:
  \   \sup_{k\in\ZZ}\left|\mathbb{E}^\upsilon_k(f)(x)\right|>2\lambda,\  \    \mathbb{S}^\upsilon(f)(x)<\varepsilon\lambda\right\}\right)
 \leq C\exp\left(-\frac{c_n}{\varepsilon^2}\right)\omega(Q_{\lambda,i}).
\end{align}
From \cite[(3.11) in p. 1466]{WY2018}, we know
 \begin{align}\label{ee10}
  &\upsilon\left(\left\{x\in Q_{\lambda,i}:\   \sup_{k\in\ZZ} \left |\mathbb{E}^\upsilon_k(f)(x)\right|>2\lambda, \   \   \mathbb{S}^\upsilon(f)(x)<\varepsilon\lambda\right\}\right)
 \leq C\exp\left(-\frac{\tilde{c}_n}{\varepsilon^2}\right)\upsilon(Q_{\lambda,i}).
\end{align}
Therefore, from Lemma~\ref{RHL} and estimate~\eqref{ee10}, we can deduce estimate~\eqref{ee9} when $\upsilon(X)=\infty$.

\smallskip

\noindent{\bf{Case~2: $0<\upsilon(X)<\infty$}}.

Let $g=f-(f)_X$. From the fact $\int_Xgd\upsilon=0$, we know that $\left\{x\in X:\   \sup_{k\in\ZZ}  \left|\mathbb{E}_k^\upsilon(g)(x)\right|>\lambda\right\} \subsetneqq X$. From the definition of the martingale  square function for dyadic cubes, $\mathbb{S}^\upsilon(f)=\mathbb{S}^\upsilon(g)$. Combining these facts, we can choose the maximal dyadic cubes $\{\widetilde{Q}_{\lambda,i}\}$ such that $\sup_{\widetilde{Q}_{\lambda,i}}|g_{\widetilde{Q}_{\lambda,i}}|>\lambda$ and use the argument of Case~1 to prove \eqref{ee8}.

This finishes the proof of Lemma~\ref{Lemma8}.
\end{proof}

\medskip

\section{Proof of  Theorem~\ref{Theorem2}}\label{sec:3}
\setcounter{equation}{0}

In this section we prove  Theorem~\ref{Theorem2}, in which $L$ is a self-adjoint operator that satisfies  Gaussian estimate  \eqref{ge} and Plancherel condition \eqref{e1.4} for some $q\in[2,\infty]$.
 Our approach is inspired by Carbery~\cite{Car} and Dappa and Trebels \cite{DT1985}.

 To prove  Theorem~\ref{Theorem2},  we need some estimates about the $L^p$-boundedness of Stein's square function associated with operators.
For $\delta>-1$, the Bochner-Riesz means  of order $\delta$ for operator $L$  are defined as
$$
S^\delta_R(L)f(x)=\int_{0}^{R^2}\Big(1-\frac{\lambda}{R^2}\Big)^\delta \dd E_L(\lambda)f(x),\  \    x\in X.
$$
We then consider  the following Stein-square function   associated to the operator $L$
$$\mathcal{G}_\delta(L)f(x)=c_\delta\left(\int_{0}^{\infty}\Big|R\frac{\partial }{\partial R}S_R^{\delta+1}(L)f(x)\Big|^2 \frac{\dd R}{R}\right)^{\frac12}.$$
\begin{lemma}\label{Steinsquare}
 Suppose that   $L$ is a self-adjoint operator that satisfies   Gaussian estimate  \eqref{ge} and Plancherel condition \eqref{e1.4} for some $q\in[2,\infty]$.
 \begin{itemize}

\item[(i)] If $2\leq p<\infty$, and
 $$
 \delta>\Big(n-\frac2q\Big)\, \Big(\frac12-\frac1p\Big)-\frac12,
 $$
 then there exists a constant $C>0$ such that
 $$\left\|\mathcal{G}_\delta(L)(f)\right\|_{L^p(d\mu)}\leq C\|f\|_{L^p(d\mu)}.$$

\item[(ii)] If $1< p\leq 2$, and
 $$
 \delta>\Big(n+1-\frac2q\Big)\, \Big (\frac1p-\frac12\Big)-\frac12,
 $$
 then there exists a constant $C>0$ such that
 $$
 \left\|\mathcal{G}_\delta(L)(f)\right\|_{L^p(d\mu)}\leq C\|f\|_{L^p(d\mu)}.
 $$

 \end{itemize}
\end{lemma}
\begin{proof}
For the case of $1<p\leq 2$, it is proved in \cite[Theorem 1.1]{CDY2013}. For the case of $2\leq p<\infty$,
the proof is essentially included in \cite{CLSY2020}. For the completeness, we give the detail of the proof  in the following.

Using dyadic decomposition, we write
$$x^\delta_+=\sum_{j\in\ZZ}2^{-j\delta }\phi(2^jx),$$
where $\phi\in C_c^\infty(1/4,1)$. Thus
\begin{align*}
\left(1-|\xi|^2\right)_+^{\delta}|\xi|^2&=\phi_0(\xi)+\sum_{k\geq1}2^{-j\delta}\phi\left (2^{j}(1-|\xi|^2)\right)\\
&\ \ +\sum_{j\geq 1}2^{-j(\delta+1)}\phi\left(2^{j}(1-|\xi|^2)\right)  2^{j}(|\xi|^2-1),
\end{align*}
where $\phi_0(\xi)=\sum_{j\leq 0}2^{-j\delta}\phi({2^j(1-|\xi|^2)})|\xi|^2$.  Let $\widetilde{\phi}(\xi)=-\phi(\xi)\xi.$
By Minkowski's inequality,
\begin{align*}
  \|\mathcal{G}_\delta(L)(f)\|_{L^p(d\mu)} & \leq \Big\|\Big(\int_{0}^{\infty}\big|\phi_0(\frac{\sqrt{L}}{R})f\big|^2\frac{\dd R}{R}\Big)^{\frac12}\Big\|_{L^p(d\mu)}\\
  &+   \sum_{j\geq 1}2^{-j\delta} \Big\|\Big(\int_{0}^{\infty} \Big| \phi\Big (2^{j} \big (1-\frac{L}{R^2}\big)\Big)  f\,\Big|^2\frac{\dd R}{R}\Big)^{\frac12}\Big\|_{L^p(d\mu)}\\
  &+ \sum_{j\geq 1}2^{-j(\delta+1)} \Big\|\Big(\int_{0}^{\infty} \Big| \widetilde{ \phi} \Big(2^{j}\big(1-\frac{L}{R^2}\big)\Big)  f\,\Big|^2\frac{\dd R}{R}\Big)^{\frac12}\Big\|_{L^p(d\mu)}.
\end{align*}
According to the $L^p$-boundedness of  square function in \cite[Proposition~4.2]{CLSY2020}, we have that for $\delta>\frac{n}{2}-\frac12-\frac1q$,
\begin{align}\label{sqq1}
\left\|\mathcal{G}_\delta(L)(f)\right\|_{L^p(d\mu)} & \leq C\bigg (1+\sum_{j\geq 1}2^{-j(\delta-\frac{n}{2}+\frac12+\frac1q)}\bigg)\,\|f\|_{L^p(d\mu)}\leq C\|f\|_{L^p(d\mu)}.
  \end{align}
By spectral theory, if $\delta>-1/2$,
\begin{align}\label{sqq2}
\left\|\mathcal{G}_\delta(L)(f)\right\|_{L^2(d\mu)} = C_\delta\|f\|_{L^2(d\mu)}.
  \end{align}
From Stein's complex interpolation with \eqref{sqq1} and \eqref{sqq2},  we see that
\begin{align*}
\|\mathcal{G}_\delta(L)(f)\|_{L^p(d\mu)} \leq C\|f\|_{L^p(d\mu)}
  \end{align*}
  provided that
  $$
  \mu >\Big(n-\frac2q\Big)\,  \Big(\frac12-\frac1p\Big)-\frac12.
  $$
This completes the proof of Lemma~\ref{Steinsquare}.
\end{proof}

\vspace{0.3cm}

Following the notation as in \cite{Car},  the fractional derivatives $m^{(\alpha)}$ of function $m$ of order $\alpha$   is defined by
\begin{eqnarray}\label{eyyy}
 \mathcal{F}(m^{(\alpha)})(\xi)=(-i\xi)^{\alpha}\mathcal{F}m(\xi).
\end{eqnarray}
For $\alpha>1/2$, the function space $L_\alpha^2(\RR_+)$  is defined as the completion of the $C^\infty$ functions of compact support in $(0,\infty)$ under the norm
\begin{align}\label{Cares1}
\|m\|^2_{L_\alpha^2}=\int_0^\infty \left|s^{\alpha+1}\Big(\frac{m(s)}{s}\Big)^{(\alpha)}\right|^2\frac{ds}{s}.
\end{align}
If multiplier function $m$ satisfies \eqref{eq1.2.3} in Theorem~\ref{Theorem2} for some $s>0$, then $m\in L_s^2(\RR_+)$. More precisely, we have the following lemma.

\begin{lemma}\label{lemmaTheorem2}
If the function $m$ on $(0,\infty)$ satisfies
$$
\sum_{k\in\ZZ}\left\|\phi m(2^k\cdot)\right\|_{W^2_{\mu}}^2<\infty
$$
for some $\mu>0$, then $m\in L_\mu^2(\RR_+)$ and
$$
\|m\|^2_{L_\alpha^2}=\int_0^\infty \left|s^{\mu+1}\Big(\frac{m(s)}{s}\Big)^{(\mu)}\right|^2\frac{ds}{s}\leq C_\mu \sum_{k\in\ZZ}\left\|\phi m(2^k\cdot)\right\|_{W^2_{\mu}}^2.
$$
\end{lemma}
\begin{proof}
The Mellin transform of F is defined as
 $$
 \mathbf{m}(F)(t):=\frac{1}{2\pi}\int_0^\infty F(\lambda)\lambda^{-1-it}d\lambda,\quad t\in \RR.
 $$
 A simple computation shows that the Mellin transform of $s^{\mu+1}\Big(\frac{m(s)}{s}\Big)^{(\mu)}$ is essentially $(1+it)^\mu$ times the Mellin transform of $m$, see also \cite[Page 53]{Car}. Note that we have the relationship between Mellin transform and Fourier transform
 $$
 \mathbf{m}(F)(t)=\mathcal F(F(e^{\cdot}\,))(t).
 $$
 Thus by Plancherel equality
   \begin{align*}
  \int_0^\infty \left|s^{\mu+1}\Big(\frac{m(s)}{s}\Big)^{(\mu)}\right|^2\frac{ds}{s}
  &= \int_{-\infty}^\infty \left| \mathbf{m} [s^{\mu+1}\Big(\frac{m(s)}{s}\Big)^{(\mu)}](t)\right|^2dt\\
  &\leq C_\mu\int_{-\infty}^\infty \left|(1-\Delta)^{\mu/2}[m(e^\cdot)](t)\right|^2dt.
     \end{align*}
Using dyadic decomposition,
  \begin{align}\label{e301.1}
  \int_0^\infty \left|s^{\mu+1}\Big(\frac{m(s)}{s}\Big)^{(\mu)}\right|^2\frac{ds}{s}
   &\leq C_\mu\sum_{j\in \ZZ}\int_{j}^{j+1} \left|(1-\Delta)^{\mu/2}[m(e^\cdot)](t)\right|^2dt\nonumber\\
   &\leq C_\mu\sum_{j\in \ZZ}\int_{j}^{j+1} \left|\phi(e^{t-j})(1-\Delta)^{\mu/2}[m(e^\cdot)](t)\right|^2dt\nonumber\\
    &=C_\mu\sum_{j\in \ZZ}\int_{0}^{1} \left|\phi(e^{t})(1-\Delta)^{\mu/2}[m(e^{\cdot-j})](t)\right|^2dt.
     \end{align}
$\phi(e^{t})$ and $(1-\Delta)^{\mu/2}$ can be considered as  pseudo-differential operators of order $0$ and $\mu$.
It follows from the symbolic calculus (see for example \cite[Theorem 2, page 237]{Stein}) that for big enough $N$
 \begin{align}\label{peudodiff}
 \phi(e^{t})(1-\Delta)^{\mu/2}=(1-\Delta)^{\mu/2}\phi(e^{t})-\sum_{k=0}^{N} \frac{i^k}{k!}\partial_\xi^k [(1+\xi^2)^{\mu/2}]\partial_t[\phi(e^{t})](x,D)+ {\rm Tail}_N(x,D)
 \end{align}
 where ${\rm Tail}_N(x,\xi)$ is symbol of order $\mu-N$.
 For the first term on the right hand of \eqref{peudodiff}, we directly have
 \begin{align}\label{e301.2}
 \sum_{j\in \ZZ}\int_{0}^{1} \left|(1-\Delta)^{\mu/2}[\phi(e^{t})m(e^{\cdot-j})](t)\right|^2dt
 \leq C_\mu \sum_{j\in \ZZ} \| \phi(e^{\cdot}) m(e^{\cdot-j})\|_{W_\mu^2}^2\leq  \sum_{k\in\ZZ}\left\|\phi m(2^k\cdot)\right\|_{W^2_{\mu}}^2.
 \end{align}
For the tail part on the right hand of \eqref{peudodiff}, we integrate by parts
  \begin{align*}
  |K_{{\rm Tail}_N(x,D)}(x,y)|=|\int_\RR e^{i(x-y)\xi} {\rm Tail}_N(x,\xi)d\xi|\leq C(1+|x-y|)^{-2}.
  \end{align*}
  Then  we apply the Cauchy-Schwarz inequality to get
    \begin{align}\label{e301.3}
    \sum_{j\in \ZZ}\int_{0}^{1} \left|{\rm Tail}_N(x,D)[m(e^{\cdot-j})](t)\right|^2dt&=
    \sum_{j\in \ZZ}\int_{0}^{1} \left|\int_\RR K_{{\rm Tail}_N(x,D)}(t,s)m(e^{s-j})ds \right|^2dt\nonumber\\
    &\leq \sum_{j\in \ZZ}\int_{0}^{1} \int_\RR (1+|t-s-j|)^{-2}|m(e^{s})|^2ds dt\nonumber\\
    &\leq C\int_\RR |m(e^{s})|^2ds \nonumber\\
    &\leq C\sum_{j\in \ZZ} \int_{j}^{j+1} \big|\phi(e^{s-j})[m(e^s)]\big|^2ds\nonumber\\
    &\leq C\sum_{j\in \ZZ} \int_{0}^{1} \big|\phi(e^{s})[m(e^{s+j})]\big|^2ds\nonumber\\
    &\leq C\sum_{k\in\ZZ}\left\|\phi m(2^k\cdot)\right\|_{W^2_{\mu}}^2.
     \end{align}
Similarly, we can estimate the terms in the summation on the right hand of \eqref{peudodiff}.
Then combining the estimates of  \eqref{e301.1}--\eqref{e301.3}, we complete the proof of Lemma~\ref{lemmaTheorem2}.
\end{proof}

\vspace{0.5cm}

Now we can prove Theorem~\ref{Theorem2}.

\begin{proof}[{\bf Proof of Theorem~\ref{Theorem2}}]

 If $m\in L_\alpha^2(\RR_+)$, it follows from \cite[Theorem 4]{Car}, we have
\begin{align}\label{fde}
 \frac{m(\lambda)}{\lambda}=C_\mu\int_{0}^{\infty}(s-\lambda)_+^{\mu-1}\Big(\frac{m(s)}{s}\Big)^{(\mu)}\dd s,\ \ \mathrm{\forall} \ \lambda>0 .
\end{align}
  Therefore, for $t>0$
   $$
  m(t\lambda)=C_\mu\int_{0}^{\infty}\frac{t\lambda}{s} \Big(1-\frac{t\lambda}{s}\Big)_+^{\mu-1}s^\mu\Big(\frac{m(s)}{s}\Big)^{(\mu)}\dd s.
  $$
  By spectral theory,
   \begin{align*}
    m(tL)f(x)= C_\mu\int_{0}^{\infty} \frac{tL}{s} \Big(1-\frac{tL}{s}\Big)_+^{\mu-1}f(x)s^\mu\Big(\frac{m(s)}{s}\Big)^{(\mu)}\dd s.
  \end{align*}
 We apply the Cauchy-Schwarz inequality to get
    \begin{align*}
   \big|m(tL)f(x)\big|&\leq C_\mu \left(\int_0^\infty \left|\frac{tL}{s} \Big(1-\frac{tL}{s}\Big)_+^{\mu-1}f(x)\right|^2\frac{ds}{s}\right)^{\frac12} \left(\int_0^\infty \left|s^{\mu+1}\Big(\frac{m(s)}{s}\Big)^{(\mu)}\right|^2\frac{ds}{s}\right)^{\frac12}\\
   &=C_\mu \left(\int_0^\infty \left|\frac{L}{s} \Big(1-\frac{L}{s}\Big)_+^{\mu-1}f(x)\right|^2\frac{ds}{s}\right)^{\frac12} \left(\int_0^\infty \left|s^{\mu+1}\Big(\frac{m(s)}{s}\Big)^{(\mu)}\right|^2\frac{ds}{s}\right)^{\frac12}\\
   &=C_\mu \mathcal{G}_{\mu-1}(L)f(x) \left(\int_0^\infty \left|s^{\mu+1}\Big(\frac{m(s)}{s}\Big)^{(\mu)}\right|^2\frac{ds}{s}\right)^{\frac12}.
     \end{align*}
 That is
   \begin{align*}
   \sup_{t>0} \big|m(tL)f(x)\big|\leq C_\mu \mathcal{G}_{\mu-1}(L)f(x) \left\|m\right\|_{L_\mu^2}.
     \end{align*}
  Thus according to Lemma~\ref{lemmaTheorem2}
 \begin{align*}
  \bigg \|\sup_{t>0}  \left|m(tL)f\right| \bigg \|_{L^p(d\mu)}\leq C_\mu \left\|\mathcal{G}_{\mu-1}(L)f\right\|_{L^p(d\mu)} \left(\sum_{k\in\ZZ}\left\|\phi m(2^k\cdot)\right\|_{W^2_{\mu}}^2\right)^{1/2}.
     \end{align*}
 Combining the estimates for Stein's square function in Lemma~\ref{Steinsquare},
we complete the proof of Theorem~\ref{Theorem2}.

\end{proof}

\section{ Some estimates for the square function  for the dyadic martingale and the Littlewood-Paley operators }\label{sec:4}
\setcounter{equation}{0}

In this section, suppose that  there exists an L-harmonic function $h$ such that the operator satisfies conditions \rm{(H-1)} and \eqref{LUGh}.
 Let  $d\upsilon=h^2d\mu$ and hence
 the space $(X, d, \upsilon)$ satisfies the doubling property by Lemma~\ref{Lemma9}.
In the following     the Hardy-Littlewood maximal function $\mathcal{M}_r^\upsilon$ in the space $(X, d, \upsilon)$ is  given by
$$
\mathcal{M}^\upsilon_r(f)(x):=\sup_{\substack{x\in B}}\left(\frac{1}{\upsilon(B)}\int_B|f(y)|^rd\upsilon(y)\right)^{\frac1r}.
$$
For simplicity we will write $\mathcal{M}^\upsilon$  in place of
 $\mathcal{M}_1^\upsilon$.

Recall that $\mathbb{S}^\upsilon$ is the square function  for the dyadic martingale
in \eqref{e2.11}.  Our main result of this section is the following proposition.

\begin{proposition}\label{Lemma2}
		Suppose that   $L$ is a self-adjoint operator that satisfies
Gaussian estimate \eqref{ge} and \eqref{e1.4} for some $q\in[2,\infty]$,  and there exists an L-harmonic function $h$ such that  the operator $L$ satisfies conditions \rm{(H-1)} and \eqref{LUGh}.   	Let  $d\upsilon=h^2d\mu$.
	Let $1< r<2$ and for every  bounded Borel function $m$ such that
	$$
	\sup_{t>0}\|\phi m(t\cdot)\|_{W^q_{s}} <\infty
	$$
	for some  $s>n/r$,  then for any $p\in (r,\infty)$,
	$$
	\mathbb{S}^\upsilon\left(h^{-1}m(L)(f)\right)(x)\leq C\Big(\sup_{t>0}\|\phi m(t\cdot)\|_{W^q_{s}}+|m(0)|\Big) \, \bigg(\sum_{j\in\ZZ}\Big|\mathfrak{M}^\upsilon_r\left(h^{-1}U_{j,L}^{2/r} (f)\right)(x)\Big|^2\bigg)^{\frac12},
	$$
	where $\mathfrak{M}_r^\upsilon:=\mathcal{M}^\upsilon_r\circ \mathcal{M}^\upsilon_r$ and $U_{j,L}^{2/r} (f)$ is given in \eqref{ujl}.
\end{proposition}

The proof of Proposition~\ref{Lemma2} will be based on  the following
result.

\begin{lemma}\label{Lemma4}
Suppose that  there exists an L-harmonic function $h$ such that the operator satisfies conditions \rm{(H-1)} and \eqref{LUGh}.
 For $1< q<\infty$
	$$
	\left\|\mathbb{D}^\upsilon_k(h^{-1}2^{-2j}L\exp(-2^{-2j}L)f)\right\|_{L^\infty({Q_k(x)},d\upsilon)}
	\leq C2^{-|j+(\log_2\delta )k|\gamma}\,\mathfrak{M}^\upsilon_q(h^{-1}f)(x).
	$$
\end{lemma}

\begin{proof}
	In the following we let $K_t$ denote the kernel of $tL\exp(-tL)$ and  $\mathcal{K}_t $  denote the kernel of  $t\mathcal{L}\exp(-t\mathcal{L})$.
	\begin{align*}
		K_t(x,y)=-t \frac{d}{dt}(T_t(x,y))=-t\frac{d}{dt}(\mathcal{T}_t(x,y)h(x)h(y))=-t\frac{d}{dt}(\mathcal{T}_t(x,y))h(x)h(y).
	\end{align*}
	Therefore,
	\begin{align*}
		tL\exp(-tL)(f)(x)&=\int_X-t\frac{d}{dt}    \big(\mathcal{T}_t(x,y)\big) h(x)h(y)f(y)d\mu(y) \\
		&=h(x)\int_X -t\frac{d}{dt}\big(\mathcal{T}_t(x,y)\big) (h^{-1}f)(y)d\upsilon(y)\\
		&=h(x)t\mathcal{L}\exp(-t\mathcal{L})(h^{-1}f)(x).
	\end{align*}
	\begin{align*}
		\mathbb{D}^\upsilon_k \left(h^{-1}2^{-2j}L\exp(-2^{-2j}L)f\right)(x) & =
		\mathbb{D}^\upsilon_k  \left(2^{-2j}\mathcal{L}\exp(-2^{-2j}\mathcal{L})(h^{-1}f)\right)(x).
	\end{align*}
	Let us fix $x$ and let $\widetilde{k}=(-\log_2\delta)k$. The problem is discussed in two cases:  $j-\widetilde{k}>10$ and  $j-\widetilde{k}\leq 10$.
	
	Case~1. $j-\widetilde{k}>10$.	For any $u\in Q_k(x)$,
	\begin{align*}
		&\mathbb{D}^\upsilon_k \left (2^{-2j}\mathcal{L}\exp(-2^{-2j}\mathcal{L})(h^{-1}f)\right)(u)\\
		&=\mathbb{E}^\upsilon_{k+1}  \left(2^{-2j}\mathcal{L}\exp(-2^{-2j}\mathcal{L})(h^{-1}f)\right)(u)-\mathbb{E}^\upsilon_k
		\left(2^{-2j}\mathcal{L}\exp(-2^{-2j}\mathcal{L} )(h^{-1}f)\right)(u).
	\end{align*}
	In the sequel of Case~1, we shall prove that there exists a $\gamma>0$ so that for any $u\in Q_k(x)$,
	\begin{align}\label{eq3.1xx}
		\left| \mathbb{E}^\upsilon_{k+1}  \left (2^{-2j}\mathcal{L}\exp(-2^{-2j}\mathcal{L})(h^{-1}f)\right)(u)\right|
		\leq C2^{-(j-\widetilde{k})\gamma}\mathcal{M}^\upsilon_q\circ\mathcal{M}^\upsilon_q(h^{-1}f)(x).
	\end{align}
 	The other term $\mathbb{E}^\upsilon_k  \left(2^{-2j}\mathcal{L}\exp(-2^{-2j}\mathcal{L})(h^{-1}f)\right)$, which is more simple, can be treated in the similar way to obtain
 	\begin{align}\label{eq3.2xx}
		\left| \mathbb{E}^\upsilon_{k} \left(2^{-2j}\mathcal{L}\exp(-2^{-2j}\mathcal{L})(h^{-1}f)\right)(u)\right|
		\leq C2^{-(j-\widetilde{k})\gamma}\mathcal{M}^\upsilon_q\circ\mathcal{M}^\upsilon_q(h^{-1}f)(x).
	\end{align}

	For this $u\in Q_k(x)\in \mathcal{D}_k$, there exists a child $Q_{k+1}\in\mathcal{D}_{k+1}$ so that $Q_{k+1}\subseteq Q_k(x)$ and $u\in Q_{k+1}$. From this $Q_{k+1}$, we next decompose the whole $X$ into three domains. To do it,
	let
	$B:=\{z\in X:dist(z,\partial Q_{k+1})\leq  \delta^{k+1}(2^{-j}/\delta^{k+1})^{1/2}   \}$,
	$ A_1:=B^c\bigcap Q_{k+1}$ and $A_2:=B^c\bigcap(Q_{k+1})^c$. Let $f_1=h^{-1}f\chi_{A_1}$,  $f_2=h^{-1}f\chi_{A_2}$ and $f_3=h^{-1}f\chi_{B}$.
	\begin{align*}
		&\mathbb{E}^\upsilon_{k+1}  \left (2^{-2j}\mathcal{L}\exp(-2^{-2j}\mathcal{L})(h^{-1}f)\right)(u)\\
		&=\frac{1}{\upsilon(Q_{k+1})} \int_{Q_{k+1}}\int_X \mathcal{K}_{2^{-2j}} (y,z)f_1(z)d\upsilon(z)d\upsilon(y)\\
		&+\frac{1}{\upsilon(Q_{k+1})} \int_{Q_{k+1}}\int_X \mathcal{K}_{2^{-2j}} (y,z)f_2(z)d\upsilon(z)d\upsilon(y)\\
		&+\frac{1}{\upsilon(Q_{k+1})} \int_{Q_{k+1}}\int_X \mathcal{K}_{2^{-2j}}(y,z)f_3(z)d\upsilon(z)d\upsilon(y)\\
		&:=I_1+I_2+I_3.
	\end{align*}
    The operator $L$ satisfies (H-1) and \eqref{LUGh}, then from Lemma~\ref{Dh},  $\int_X \mathcal{T}_{2^{-2j}}(x,y)d\upsilon(y)=1$.
	Then from \cite[proposition~2.2]{BDY2012},  we have
	$\int_{X}\mathcal{K}_{2^{-2j}}(y,z)d\upsilon(z)=0$. Therefore,  for any $y\in A_1$, there exist constants $C,c_2>0$ and $\alpha>0$ such that
	\begin{align*}
		&\bigg|\int_{Q_{k+1}}\mathcal{K}_{2^{-2j}}(y,z) d\upsilon(y)\bigg|=\bigg |\int_{Q_{k+1}^c}\mathcal{K}_{2^{-2j}}(y,z)d\upsilon(y)\bigg|\\
		&\leq C\int_{Q_{k+1}^c}
		\upsilon(B(z,2^{-j}))^{-1}\exp\Big(-\frac{d(y,z)^2}{c_22^{-2j}}\Big)d\upsilon(y)\\
		&\leq C\int_{d(y,z)\geq \delta^{k+1}(2^{-j}/\delta^{k+1})^{1/2}}\upsilon(B(z,2^{-j}))^{-1}\exp\Big(-\frac{d(y,z)^2}{c_22^{-2j}}\Big) \, d\upsilon(y)\\
		&\leq C\exp\left(-c_2^{-1}\delta2^{j-1}\delta^k\right)\leq C2^{-\alpha(j-\widetilde{k})}.
	\end{align*}
	
	Therefore, 	from $Q_{k+1}\subseteq Q_{k}(x)$ and doubling property, we have
	\begin{align*}
		I_1&=\bigg |\frac{1}{\upsilon(Q_{k+1})} \int_{Q_{k+1}}\int_{A_1} \mathcal{K}_{2^{-2j}} (y,z)f_1(z)d\upsilon(z)d\upsilon(y) \bigg|\\
		&\leq C2^{-\alpha(j-\widetilde{k})}\frac{1}{\upsilon(Q_{k+1})} \int_{A_1}
		\left |h^{-1}(z)f(z)\right|d\upsilon(z)\\
		&\leq C2^{-(j-\widetilde{k})\alpha}\mathcal{M}^\upsilon(h^{-1}f)(x).
	\end{align*}
	For the second term,
	\begin{align*}
		&\bigg|\int_{A_2}f(z)\mathcal{K}_{2^{-2j}}(y,z)d\upsilon(z)\bigg|\\
		&\leq C\int_{d(y,z)\geq \delta^{k+1}(2^{-j}/\delta^{k+1})^{1/2}}\left|h^{-1}(z)f(z)\right|\upsilon(B(z,2^{-j}))^{-1}\exp\Big(-\frac{d(y,z)^2}{c_22^{-2j}}\Big)\, d\upsilon(z)\\
		&\leq C\exp\Big(-\frac{\delta}{2c_2}2^{j}\delta^k\Big)\int_{X}  \left|h^{-1}(z)f(z)\right|\upsilon(B(z,2^{-j}))^{-1}\exp\Big(-\frac{d(y,z)^2}{2c_2\cdot2^{-2j}}\Big)\, d\upsilon(z)\\
		&\leq C2^{-(j-\widetilde{k})\alpha}\mathcal{M}^\upsilon(h^{-1}f)(y),
	\end{align*}
which combines $Q_{k+1}\subseteq Q_{k}(x)$ and doubling property can deduce
	\begin{align*}
		I_2&\leq C2^{-(j-\widetilde{k})\alpha}\frac{1}{\upsilon(Q_{k+1})} \int_{Q_{k+1}}\mathcal{M}^\upsilon(h^{-1}f)(y) d\upsilon(y)\\
		&\leq C2^{-(j-\widetilde{k})\alpha}\mathcal{M}^\upsilon\circ\mathcal{M}^\upsilon(h^{-1}f)(x).
	\end{align*}
From the definition of $B$, we know that $B\subseteq 2Q_{k+1}\subseteq 2Q_{k}(x)$.
	By the property of dyadic cubes (see Lemma~2.1), there exists  $\rho>0$ such that $\upsilon(B)\leq (2^{-j}/\delta^{k+1})^{\rho}\upsilon(Q_{k+1})$, which combines the H\"older inequality and doubling property imply that for $1<q<\infty$,
	\begin{align*}
		I_3&\leq \bigg|\frac{1}{\upsilon(Q_{k+1})} \int_B\int_{Q_{k+1}} \mathcal{K}_{2^{-2j}} (y,z)h^{-1}f(z)d\upsilon(y)d\upsilon(z)\bigg|\\
		&\leq \frac{C}{\upsilon(Q_{k+1})}\int_B|(h^{-1}f)(z)|d\upsilon(z)\\
		&\leq C\left(\frac{\upsilon(B)}{\upsilon(Q_{k+1})}\right)^{\frac1{q'}}\left(\frac{\upsilon(2Q_{k}(x))}{\upsilon(Q_{k+1})}\right)^{\frac1q}
		\bigg(\fint_{2Q_k(x)}    \left |(h^{-1}f)(z)\right|^qd\upsilon(z)  \bigg)^{1/q}\\
		&\leq C2^{-\frac{(j-\widetilde{k})\rho}{q'}}\mathcal{M}^\upsilon_q(h^{-1}f)(x).
	\end{align*}
	Hence, for any $u\in Q_k(x)$, there exists a $\gamma=\min\{\alpha,\rho/q'\}>0$ such that \eqref{eq3.1xx} holds, i.e.
	\begin{align*}
		&\left|\mathbb{E}^\upsilon_{k+1}(2^{-2j}\mathcal{L}\exp(-2^{-2j}\mathcal{L})(h^{-1}f))(u)\right|\leq C2^{-(j-\widetilde{k})\gamma}\mathcal{M}^\upsilon_q\circ\mathcal{M}^\upsilon_q(h^{-1}f)(x).
	\end{align*}
	Similarly, we can obtain \eqref{eq3.2xx}.	Hence, combining \eqref{eq3.1xx} and \eqref{eq3.2xx}, we obtain
	\begin{align*}
		&\left\|\mathbb{D}^\upsilon_k  \left (2^{-2j}\mathcal{L}\exp(-2^{-2j}\mathcal{L}\right)(h^{-1}f))\right\|_{L^\infty({Q_k(x)},d\upsilon)}\leq C2^{-(j-\widetilde{k})\gamma}\,\mathfrak{M}^\upsilon_q(h^{-1}f)(x).
	\end{align*}
	
	Case~2. $j-\widetilde{k}\leq 10$. For $u\in Q_k(x)$, there exists a $Q_{k+1}\in \mathcal{D}_{k+1}$ so that $Q_{k+1}\subseteq Q_k(x)$ and $u\in Q_{k+1}$.
	\begin{align*}
		&  \left|\mathbb{D}^\upsilon_k  \left (2^{-2j}\mathcal{L}\exp(-2^{-2j}\mathcal{L})(h^{-1}f)\right)(u)\right| \\
		&= \left |\mathbb{E}^\upsilon_{k+1}  \left(2^{-2j}\mathcal{L}\exp(-2^{-2j}\mathcal{L})(h^{-1}f)\right)(u)-\mathbb{E}^\upsilon_k \left(2^{-2j}\mathcal{L}\exp(-2^{-2j}\mathcal{L})(h^{-1}f)\right)(u)\right|\\
		& \leq \bigg |\int_X \left(\frac{1}{\upsilon(Q_{k+1})}\int_{Q_{k+1}}
		\Big( \mathcal{K}_{2^{-2j}}(y,z)-\mathcal{K}_{2^{-2j}}(u,z)\Big)\, d\upsilon(y)(h^{-1}f)(z)\right)d\upsilon(z)\bigg |\\
		&+\bigg |\int_X \left(\frac{1}{\upsilon(Q_{k}(x))}\int_{Q_{k}(x)}
		\Big(\mathcal{K}_{2^{-2j}}(y,z)-\mathcal{K}_{2^{-2j}}(u,z)\Big)\, d\upsilon(y)(h^{-1}f)(z)\right)d\upsilon(z)\bigg|\\
		&=:H_1+H_2.
	\end{align*}
	According to the property (5) of dyadic cubes, $Q_{k+1}\subseteq B(x_{k},C\delta^{k+1})$ for some $x_k\in X$.
	For any $y\in Q_{k+1}$,
	$$d(u,y)\leq d(x_k,u)+d(x_k,y)\leq c_n\delta^{k}\leq C2^{-\widetilde{k}}.$$
	So by the continuity of $\mathcal{K}_{2^{-2j}}$ (see estimate~\eqref{Tf2}), there exist constants $\alpha>0$  and $c_4>0$ such that
	\begin{align*}
		|\mathcal{K}_{2^{-2j}}(y,z)-\mathcal{K}_{2^{-2j}}(u,z)|&\leq C\upsilon(B(z,2^{-j}))^{-1}\exp\Big(-\frac{d(y,z)^2}{c_42^{-2j}}\Big) \, \Big(\frac{d(u,y)}{2^{-j}}\Big)^{\alpha}\\
		&\leq C2^{(j-\widetilde{k})\alpha}\upsilon(B(z,2^{-j}))^{-1}\exp\Big(-\frac{d(y,z)^2}{c_42^{-2j}}\Big).
	\end{align*}
Combining $Q_{k+1}\subseteq Q_{k}(x)$, we have
	\begin{align*}
		H_1&\leq  C2^{(j-\widetilde{k})\alpha}\int_X\fint_{Q_{k+1}}\upsilon(B(z,2^{-j}))^{-1}
		\exp\left(-\frac{d(y,z)^2}{c_42^{-2j}}\right)  \left|(h^{-1}f)(z)\right|
		d\upsilon(y)d\upsilon(z)\\
		&\leq C2^{(j-\widetilde{k})\alpha}\frac{\upsilon({Q_{k}(x)})}{\upsilon({Q_{k+1}})}\fint_{Q_{k}(x)}\mathcal{M}^\upsilon(h^{-1}f)(y)d\upsilon(y)\\
		&\leq C2^{(j-\widetilde{k})\alpha}\mathcal{M}^\upsilon\circ \mathcal{M}^\upsilon(h^{-1}f)(x).
	\end{align*}
	Similarly, we can obtain that
	$$H_2\leq C2^{(j-\widetilde{k})\alpha}\mathcal{M}^\upsilon\circ \mathcal{M}^\upsilon(h^{-1}f)(x).$$
	Combining Case~1 and Case~2, we obtain
$$
	\left\|\mathbb{D}^\upsilon_k  \left(h^{-1}2^{-2j}L\exp(-2^{-2j}L)f\right)\right\|_{L^\infty({Q_k(x)},d\upsilon)}
	\leq C2^{-|j+(\log_2\delta )k|\gamma}\,\mathfrak{M}^\upsilon_q(h^{-1}f)(x).
$$
This completes the proof of Lemma~\ref{Lemma4}.
\end{proof}

\vspace{0.5cm}

Next, let us prove Proposition~\ref{Lemma2}.

\begin{proof}[Proof of Proposition~\ref{Lemma2}]
	Let $s>n/r$ and $k>s/2+1$. Let $\Psi$ be defined as Subsection~\ref{subsec:2.1}.  Let $k\in \NN$. By the theory of spectral theory,
	$$m(L)f=c_{\Psi}\sum_{j\in\ZZ}2^{-2j}L\exp(-2^{-2j}L)\int_{2^{-j-1}}^{2^{-j}}(t^2L)^k\exp(-t^2L)\Psi(t\sqrt{L})m(L)f\frac{dt}{t},$$
	where the summation is in the sense of $L^2(d\mu)$. Then
	\begin{align*}
		&\mathbb{D}^\upsilon_k(h^{-1}m(L)(f))(x)\\
		&=\mathbb{D}^\upsilon_k\bigg(h^{-1}\sum_{j\in\ZZ}2^{-2j}L\exp(-2^{-2j}L)\int_{2^{-j-1}}^{2^{-j}}(t^2L)^k
		\exp(-t^2L)\Psi(t\sqrt{L})m(L)f\frac{dt}{t}\bigg)\, (x)\\
		&=\sum_{j\in\ZZ}\mathbb{D}^\upsilon_k\bigg(h^{-1}2^{-2j}L\exp(-2^{-2j}L)\int_{2^{-j-1}}^{2^{-j}}(t^2L)^k
		\exp(-t^2L)\Psi(t\sqrt{L})m(L)f\frac{dt}{t}\bigg)\,(x),
	\end{align*}
	where  we use $d\upsilon=h^2d\mu$ and the dominated convergence theorem in the procedure of interchange order of the summations.
	
	From Lemma~\ref{Lemma4} and Minkowski's inequality, for $1<r<2$, we have
	\begin{align*}
		&\left\|\mathbb{D}^\upsilon_k \left(h^{-1}m(L)(f)\right)\right\|_{L^\infty(Q_k(x),d\upsilon)}\\
		&\leq \sum_{j\in\ZZ}\bigg \|\mathbb{D}^\upsilon_k\bigg(h^{-1}2^{-2j}L\exp(-2^{-2j}L)\int_{2^{-j-1}}^{2^{-j}}(t^2L)^k
		\exp(-t^2L)\Psi(t\sqrt{L})m(L)f\frac{dt}{t}\bigg)\,\bigg\|_{L^\infty(Q_k(x),d\upsilon)}\\
		&\leq C\sum_{j\in\ZZ}2^{-|j+(\log_2\delta)k|\gamma}\,\mathfrak{M}^\upsilon_r\bigg(h^{-1}\int_{2^{-j-1}}^{2^{-j}}(t^2L)^k
		\exp(-t^2L)\Psi(t\sqrt{L})m(L)f\frac{dt}{t}\bigg)\,(x).
	\end{align*}
	Following by Lemma~\ref{Lemma3} and the Cauchy-Schwarz inequality, we have
	\begin{align*}
		&\bigg(\int_{2^{-j-1}}^{2^{-j}}(t^2L)^k\exp(-t^2L)\Psi(t\sqrt{L})m(L)f(x)\frac{dt}{t}\bigg)^{\frac12}\\
		&\leq C\Big (\sup_{t>0}\|\phi m(t\cdot)\|_{W^q_s}+|m(0)|\Big)\, U_{j,L}^{2/r}(f)(x),
	\end{align*}
	which gives
	\begin{align*}
		&\left\|\mathbb{D}^\upsilon_k  \left (h^{-1}m(L)(f)\right)\right\|_{L^\infty(Q_k(x),d\upsilon)}\\
		&\leq C\Big(\sup_{t>0}\|\phi m(t\cdot)\|_{W^q_{s}}+|m(0)|\Big)\,   \sum_{j\in\ZZ}2^{-|j+(\log_2\delta)k|\gamma}\, \mathfrak{M}^\upsilon_r \left (h^{-1}U_{j,L}^{2/r} (f)\right)(x).
	\end{align*}
	By using the Cauchy-Schwarz inequality and interchanging the order of summations, we have
	\begin{align*}
		&\Bigg(\sum_{k\in\ZZ}  \left \|\mathbb{D}^\upsilon_k \left(h^{-1}m(L)(f)\right)\right\|_{L^\infty(Q_k(x),d\upsilon)}^2\Bigg)^{\frac12} \\
		&\leq C \Big(\sup_{t>0}  \left\|\phi m(t\cdot)\right\|_{W^q_{s}}+|m(0)|\Big)\,  \Bigg(\sum_{k\in\ZZ}\Big|\sum_{j\in\ZZ}2^{-|j+(\log_2\delta)k|\gamma} \, \mathfrak{M}^\upsilon_r  \left (h^{-1}U_{j,L}^{2/r} (f)\right)(x)\Big|^2\Bigg)^{\frac12}\\
		&\leq  C\Big(\sup_{t>0}  \left\|\phi m(t\cdot)\right\|_{W^q_{s}}+|m(0)|\Big)\,  \Bigg(\sum_{j\in\ZZ}\sum_{k\in\ZZ}2^{-2|j+(\log_2\delta)k|\gamma} \,\Big|\mathfrak{M}^\upsilon_r  \left(h^{-1}U_{j,L}^{2/r} (f)\right)(x)\Big|^2\Bigg)^{\frac12}\\
		&\leq C\Big(\sup_{t>0}  \left\|\phi m(t\cdot)\right\|_{W^q_{s}}+|m(0)|\Big)\,  \Bigg(\sum_{j\in\ZZ}\Big|\mathfrak{M}^\upsilon_r  \left (h^{-1}U_{j,L}^{2/r} (f)\right)(x)\Big|^2\Bigg)^{\frac12}.
	\end{align*}
	This completes the proof of Proposition~\ref{Lemma2}.
\end{proof}

\medskip

\section{Proof of  Theorem~\ref{Theorem1}}\label{sec:5}
\setcounter{equation}{0}

Let $h$ be the function defined in Lemma~\ref{Dh} and  $
d\upsilon(x)=h(x)^2d\mu(x),
$  as Subsection~\ref{subsec:2.2}. Let $d\omega=h^{p-2}d\upsilon$. The problem is discussed in two cases: $\upsilon(X)=\infty$ and $0<\upsilon(X)<\infty$.

\noindent{\bf{Case~1. $\upsilon(X)=\infty$}}.  Let $f\in \mathcal{S}$.
\begin{align*}
  \Big \|\sup_{1\leq i\leq N}|m_i(L)f|\ \Big\|_{L^p(d\mu)}^p
  &\leq \Big \|\sup_{1\leq i\leq N}  \left |h^{-1}m_i(L)f  \right|\ \Big\|_{L^p(h^{p-2}d\upsilon)}^p\\
  &=\Big  \|\sup_{1\leq i\leq N}  \left|h^{-1}m_i(L)f\right|\ \Big  \|_{L^p(d\omega)}^p\\
  &=2^pp\int_{0}^{\infty} \lambda^{p-1}
  \omega\left(\left\{x\in X:\sup_{1\leq i\leq N}  \left|h^{-1}m_i(L)f(x)\right|>2\lambda\right\}\right) d\lambda.
\end{align*}
We select    an $r_0$ such that $r_0>r$ and $r_0<\min\{2,p\}$, then according to Lemma~\ref{Lemma2}, there exists a constant $C_{r_0}$ such that
\begin{align*}
\mathbb{S}^\upsilon  \left(h^{-1}m_i(L)(f)\right)(x)&\leq C_{r_0}  \Big(\sup_{t>0}\|\phi m(t\cdot)\|_{W^q_{s}}+|m(0)|\Big)\,  \bigg(\sum_{j\in\ZZ}\Big|\mathfrak{M}^\upsilon_{r_0}  \left(h^{-1}U_{j,L}^{2/{r_0}} (f)\right) (x)\Big|^2\bigg)^{\frac12}\\
&\leq  C_{r_0}B\bigg(\sum_{j\in\ZZ}\Big|\mathfrak{M}^\upsilon_{r_0}  \left(h^{-1}U_{j,L}^{2/{r_0}} (f)\right) (x)\Big|^2\bigg)^{\frac12}.
\end{align*}
Let  $\varepsilon_N=\big(C\sqrt{\log(1+N)}\,\big)^{-1}$, then
\begin{align*}
  &\left\{x\in X:\sup_{1\leq i\leq N}  \left|h^{-1}m_i(L)(f)(x)\right|>2\lambda\right\}\subseteq I_\lambda^1\cup I_\lambda^2,
\end{align*}
where
$$
  I_\lambda^1  = \left\{x\in X:\sup_{1\leq i\leq N}|h^{-1}m_i(L)(f)(x)|>2\lambda,\   \   \Big(\sum_{j\in\ZZ}\big|\mathfrak{M}^\upsilon_{r_0}  \left(h^{-1}U_{j,L}^{2/{r_0}} (f)\right) (x)\big|^2\Big)^{\frac12}\leq \frac{\varepsilon_N\lambda}{C_{r_0}B} \right\},
$$
and
$$
  I_\lambda^2=
 \left\{x\in X:\Big(\sum_{j\in\ZZ}\big|\mathfrak{M}^\upsilon_{r_0}  \left(h^{-1}U_{j,L}^{2/{r_0}} (f)\right) (x)\big|^2\Big)^{\frac12}>  \frac{\varepsilon_N\lambda}{C_{r_0}B} \right\}.
$$


For the first term $I_\lambda^1$,
$$
    I_\lambda^1\subseteq \bigcup_{i=1}^N\left\{x\in X:   \left|h^{-1}m_i(L)(f)(x)\right|>2\lambda, \    \  \mathbb{S}^\upsilon \left(h^{-1}m_i(L)(f)\right)(x)\leq \varepsilon_N \lambda\right\}.
$$
Therefore,
\begin{align*}
 &\omega\left(\left\{x\in X:\   \sup_{1\leq i\leq N}  \left|h^{-1}m_i(L)(f)(x)\right|>2\lambda,  \   \   \bigg(\sum_{j\in\ZZ}\left|\mathfrak{M}^\upsilon_{r_0}(h^{-1}U_{j,L}^{2/{r_0}} (f))(x)\right|^2\bigg)^{\frac12}\leq  \frac{\varepsilon_N\lambda}{C_{r_0}B}\right\}\right)\\
 &\leq\sum_{i=1}^{N}\omega\left(\left\{x\in X:\   \sup_{k\in\ZZ} \left|\mathbb{E}_k^\upsilon(h^{-1}m_i(L)(f))(x)\right|>2\lambda, \  \    \mathbb{S}^\upsilon\left(h^{-1}m_i(L)(f)\right)(x)\leq \varepsilon_N\lambda\right\}\right).
\end{align*}
From  $m_i(L)f\in L^{2}(d\mu)$, we have $h^{-1}m_i(L)f\in L^2(d\upsilon)$ . From Lemma~\ref{Lemma5}, we have
\begin{align*}
&\omega\left(\left\{x\in X:\,   \left|h^{-1}m_i(L)(f)\right|>2\lambda,\   \     \mathbb{S}^\upsilon  \left(h^{-1}m_i(L)(f)\right)(x)\leq \varepsilon_N\lambda\right\}\right)\\
 &\leq C\exp\left(-\frac{c_n}{\varepsilon_N^2}\right)\omega\left(\left\{x\in X:\   \sup_{k\in \ZZ}  \left|\mathbb{E}^\upsilon_k\left(h^{-1}m_i(L)f\right)(x)\right|>\lambda\right\}\right).
\end{align*}
Therefore, by the $L^p$-boundedness of spectral multipliers and the weighted $L^p$-boundedness of Hardy-Littlewood maximal operators, for $\omega\in A_{p/r}(d\upsilon)$, we have
\begin{align}\label{b1}
 &\left(p\int_{0}^{\infty}\lambda^{p-1}\omega(I_\lambda^1)d\lambda\right)^{\frac1p}\nonumber\\
 &\leq C\left(\exp\bigg(-\frac{c_n}{\varepsilon_N^2}\bigg)\right)^{\frac1p} \left(\sum_{i=1}^{N}p\int_{0}^{\infty}\lambda^{p-1}\omega\left(\left\{x\in X:\sup_{k\in \ZZ} \left|\mathbb{E}^\upsilon_k\left(h^{-1}m_i(L)f\right)(x)\right|>\lambda\right\}\right)  d\lambda\right)^{\frac1p}\nonumber\\
 &\leq  C\left(\exp\bigg(-\frac{c_n}{\varepsilon_N^2}\bigg)\right)^{\frac1p}
 \bigg(\sum_{i=1}^{N}\left\|\mathcal{M}^\upsilon  \left(h^{-1}m_i(L)(f)\right)\right\|_{L^p(d\omega)}^p\bigg)^{\frac1p}\nonumber\\
 &\leq C \left(\exp\bigg(-\frac{c_n}{\varepsilon_N^2}\bigg)\right)^{\frac1p}
 \bigg(\sum_{i=1}^{N}  \left\|h^{-1}m_i(L)(f)\right\|_{L^p(d\omega)}^p\bigg)^{\frac1p}\nonumber\\
 &=C\left(\exp\bigg(-\frac{c_n}{\varepsilon_N^2}\bigg)\right)^{\frac1p}\bigg(\sum_{i=1}^{N}\left\|m_i(L)(f)\right\|_{L^p(d\mu)}^p\bigg)^{\frac1p}\nonumber\\
 &\leq CB\left(N\exp\bigg(-\frac{c_n}{\varepsilon_N^2}\bigg)\right)^{\frac1p}\|f\|_{L^p(d\mu)}\leq CB\|f\|_{L^p(d\mu)}.
\end{align}
For the second term $I_\lambda^2$,  by a change of variable, we have
\begin{align}\label{b2}
  \left(p\int_{0}^{\infty}\omega(I_\lambda^2)\lambda^{p-1}d\lambda\right)^{\frac1p}&\leq CB\varepsilon_N^{-1}\bigg\|\Big(\sum_{j\in\ZZ}\Big|   \mathfrak{M}^\upsilon_{r_0}  \left(h^{-1}U_{j,L}^{2/{r_0}} (f)\right)\Big|^2\bigg)^{\frac12}\bigg\|_{L^p(d\omega)}.
\end{align}
To continue, let us recall the Fefferman-Stein vector-valued maximal inequality  from \cite{FS1971,GLY2009} and its variant from \cite[p.9 (8)]{BBD2020}. For $0<p<\infty$, $0<r<\min\{2,p\}$ and $\omega\in A_{p/r}(d\upsilon)$, then for any sequence of measurable functions $\{f_j\}$
\begin{align}\label{SFI}
  &\bigg\|\Big(\sum_{j\in\ZZ} \left|\mathfrak{M}^\upsilon_r(f_j)\right|^2\Big)^{\frac12}\bigg\|_{L^p(d\omega )}\leq C\bigg\| \Big(\sum_{j\in\ZZ}|f_j|^2\Big)^{\frac12}\bigg\|_{L^p(d\omega )}.
\end{align}

  According to  \eqref{SFI} and the $L^p$-boundedness of Littlewood-Paley G-function  (see estimate~\eqref{GLI}) and also $\omega\in A_{p/r}(d\upsilon)\subseteq A_{p/r_0}(d\upsilon) $, we have
   \begin{align}\label{b5}
\nonumber&\bigg\|\Big(\sum_{k\in\ZZ}   \left|\mathfrak{M}^\upsilon_{r_0}  \left(h^{-1}U_{k,L}^{2/{r_0}}(f)\right)\right|^2\Big)^{\frac12}\bigg\|_{L^p(d\omega)}\\\nonumber
  &\leq C_{p,r}\bigg\|\Big(\sum_{j\in\ZZ} \left|h^{-1}U_{j,L}^{2/{r_0}}(L)f\right|^2\Big)^{\frac12}\bigg\|_{L^p(d\omega)}\\\nonumber
  &= C_{p,r}\bigg\|\Big(\sum_{j\in\ZZ}\left |U_{j,L}^{2/{r_0}}(L)f\right|^2\Big)^{\frac12}\bigg\|_{L^p(d\mu)}\\
  &= C_{p,r}\left\|G_L^{2/{r_0}}(f)\right\|_{L^p(d\mu)}\leq C_{p,r}\|f\|_{L^p(d\mu)},
  \end{align}
  which combining estimate~\eqref{b2} yields
  \begin{align}\label{b3}
 \nonumber\left(p\int_{0}^{\infty}\omega(I_\lambda^2)\lambda^{p-1}d\lambda\right)^{\frac1p}&\leq C_{p,r}B\varepsilon_N^{-1}\|f\|_{L^p(d\mu)}\\
 &\leq C_{p,r}B\sqrt{\log(N+1)}\|f\|_{L^p(d\mu)}.
  \end{align}

%

Finally, combining estimates~\eqref{b1}, \eqref{b3} and using density argument, we obtain Theorem~\ref{Theorem1} when $\upsilon(X)=\infty$.

\noindent {\bf{Case~2.}} $0<\upsilon(X)<\infty$. Let $(h^{-1}m_i(L)f)_X=\upsilon(X)^{-1}\int_Xh^{-1}m_i(L)f d\upsilon.$
\begin{align*}
  \bigg\|\sup_{1\leq i\leq N}  \left|m_i(L)f\right|\bigg \|_{L^p(d\mu)}
  &\leq C\bigg\|\sup_{1\leq i\leq N}  \left |h^{-1}m_i(L)f\right|\bigg \|_{L^p(d\omega)}\\
  &\leq  C \bigg \|\sup_{1\leq i\leq N}   \left| \left(h^{-1}m_i(L)f\right)_X\right|\ \bigg\|_{L^p(d\omega)}\\
  &\ \ +C\bigg\|\sup_{1\leq i\leq N}  \left |h^{-1}m_i(L)f-\left(h^{-1}m_i(L)f\right)_X\right|\  \bigg \|_{L^p(d\omega)}.
\end{align*}
For the first term, from  H\"older's inequality,
\begin{align}\label{eexx5}
  \nonumber\left| \left(h^{-1}m_i(L)f\right)_X\right|
  &\leq \upsilon(X)^{-1}\left(\int_X  \left|m_i(L)f(y)\right|^ph^{-2}d\upsilon(y)\right)^{\frac1p}\left(\int_Xh(y)^{(\frac2p-1)p'}d\upsilon(y)\right)^{\frac1{p'}}\\
  &= \upsilon(X)^{-1}\left(\int_X  \left|m_i(L)f(y) \right|^pd\mu(y)\right)^{\frac1p} \left(\int_X h(y)^{-\frac{p-2}{p-1}}d\upsilon(y)\right)^{\frac{1}{p'}}.
\end{align}
If $h^{p-2}\in A_{p/r}(d\upsilon)$,  from $A_{p/r}(d\upsilon)\subseteq A_{p}(d\upsilon)$, then
\begin{align*}
\left(\int_B h(y)^{-\frac{p-2}{p-1}}d\upsilon(y)\right)^{\frac{1}{p'}}\left(\int_B h(y)^{p-2}d\upsilon(y)\right)^{\frac1p}\leq C\upsilon(B),\ \ \ \forall B\subseteq X.
\end{align*}
Note that $d\omega=h^{p-2}d\upsilon$ and $0<\upsilon(X)<\infty$. Using approximation and the fact $\omega$ is a weight, which is a locally integral function, then we can obtain $0<\omega(X)<\infty$ and
\begin{align*}
 \left(\int_X h^{-\frac{p-2}{p-1}}(y)d\upsilon(y)\right)^{\frac{1}{p'}}\leq C\omega(X)^{-\frac1p} \upsilon(X),
\end{align*}
which with \eqref{eexx5} and  Proposition~\ref{thmbydo}, implies that there exists a constant $C$\,(independent of $i$ and $N$) such that
\begin{align}\label{eexx2}
\Big\|\sup_{1\leq i\leq N}|(h^{-1}m_i(L)f)_X|\ \Big\|_{L^p(d\omega)}\leq CB\|f\|_{L^p(d\mu)}.
\end{align}

For the second term,  we can treat it as the case $\upsilon(X)=\infty$ to use Lemma~\ref{Lemma8} and Proposition~\ref{Lemma2} and then obtain
\begin{align}\label{eexx1}
\bigg \|\sup_{1\leq i\leq N}  \left|h^{-1}m_i(L)f- \left(h^{-1}m_i(L)f\right)_X\right|\ \bigg\|_{L^p(d\omega)}\leq C_{p,r}B\sqrt{\log{(N+1)}}  \left\|f\right\|_{L^p(d\mu)}.
\end{align}
Therefore, from \eqref{eexx2} and \eqref{eexx1}, we conclude that there exists constant $C_{p,r}$ such that
\begin{align*}
 \bigg \|\sup_{1\leq i\leq N} \big|  m_i(L)f\big| \,\bigg \|_{L^p(d\mu)}  &\leq C_{p,r}B\sqrt{\log{(N+1)}}\left\|f\right\|_{L^p(d\mu)}.
\end{align*}
Combining Case~1 and Case~2, we finish the proof of Theorem~\ref{Theorem1}.

    \hfill{}$\Box$

\begin{remark}
Fix $s\geq 0$ and $q\geq1$, then there exists a nonzero function $\eta\in C_c^\infty(\RR)$ with support $\mathrm{supp}\,\eta\subseteq[1/8,8]$ such that
$$
    \sup_{k\in\ZZ}  \left\|\phi m(2^k\cdot)\right\|_{W^q_{s}}\leq  \sup_{t>0} \left\|\phi m(t\cdot)\right \|_{W^q_{s}}\leq C_{s,\phi,\eta}\sup_{k\in\ZZ}  \left\|\eta m(2^k\cdot)\right\|_{W^q_{s}}.
$$

\end{remark}
    \medskip
%
%
%

\medskip

\section{Proof of Theorem~\ref{Theorem3} and Theorem~\ref{Theorem8}}\label{sec:6}
\setcounter{equation}{0}

\subsection{Proof of Theorem~\ref{Theorem3}} \label{subsec:6.1}
Before giving the proof of Theorem~\ref{Theorem3}, we first recall a property of the Rademacher function in \cite{G2014} and the titling lemma in \cite{CGHS2005}.
We denote $\{r_j(s)\}$ the Rademacher function. Let $0<p<\infty$. $\{a_j\}$ and $\{c_j\}$ are real sequences. Then
$$
    \bigg(\sum_{j\in\ZZ} |a_j+ic_j|^2\bigg)^{\frac12}\approx \bigg(\int_{0}^{1}\Big(\sum_{j\in\ZZ}r_j(s)(a_j+ic_j)\Big)^pds\bigg)^\frac1p.
$$

\smallskip
\begin{lemma}\label{Lemma6}
	Let $N > 0$ and let $F$ be a set of integers with cardinality Car$(F)\leq  2^N$. Then
	there exists a set $B=\{b_i\}_{i\in\ZZ}$ integers, so that
	
\begin{itemize}
\item[\rm (i)] the sets $\{b_i\} + F$ are pairwise disjoint,
	
\item[\rm (ii)] $\ZZ=\bigcup_{|\ell|\leq 4^{N+1}}B+\ell$,
	
\item[\rm (iii)] $b_i\in  [i4^{N+1}, (i + 1)4^{N+1})$.
\end{itemize}
\end{lemma}


We begin to prove  Theorem~\ref{Theorem3}.

\noindent{\rm{\bf {Proof of (i)}}}. We firstly define a sequence of sets by
$$
    F_j=\left\{k\in\ZZ:\omega^{*}(2^{2^j})<\omega(k)\leq \omega^{*}(2^{2^{j-1}})\right\},\quad  j\in \NN^+
$$
and
$$
    F_0=\left\{k\in\ZZ:\omega^{*}(2^{2})<\omega(k)\leq \omega^{*}(0)\right\}.
$$
Let $\phi\in C_c^\infty([1/2,2])$. Based on the sets $\{F_j\}$, $m(L)f$ is decomposed into a sequence of operators as
\begin{align*}
  m(L)f-m(0)f= \sum_{k\in\ZZ}m(L)\phi(2^{-k}L)f=\sum_{j\geq0}\widetilde{m}_j(L)f,
\end{align*}
where $\widetilde{m}_j(t)=\sum_{k\in F_j}m(t)\phi(2^{-k}t).$

For $F_j$, the cardinality of $F_j$ satisfies Car$(F_j)\leq 2^{2^j}$, following by the titling lemma~\ref{Lemma6}, there exist a sequence of integer sets $B=\{b_i\}_{i\in\ZZ}$ so that
\begin{align*}
  M_{\widetilde{m}_j,L}^{\mathrm{dyad}}(f)=\sup_{k\in\ZZ} \left|\widetilde{m}_j(2^k L)(f)(x)\right| =\sup_{|\alpha|\leq 4^{N+1}}\sup_{i\in\ZZ}  \left|\widetilde{m}_j(2^{b_i+\alpha} L)(f)(x)\right|.
\end{align*}
Using the property of $L^p$ norm of Rademacher functions $\{r_i(s)\}_{i\in\ZZ}$, we have
\begin{align*}
  & \bigg\| \sup_{|\alpha|\leq 4^{N+1}}\sup_{i\in\ZZ}   \left|\widetilde{m}_j(2^{b_i+\alpha} L)(f)\right|\ \bigg\|_{L^p(d\mu)}\\
  &\leq \bigg\|\sup_{|\alpha|\leq 4^{N+1}}   \Big (\sum_{i\in\ZZ} \left|\widetilde{m}_j(2^{b_i+\alpha} L)(f)\right|^2  \Big)^{\frac12}\bigg \|_{L^p(d\mu)}\\
  &\leq  \bigg\|\sup_{|\alpha|\leq 4^{N+1}}\bigg(\int_{0}^{1}   \Big|\sum_{i\in\ZZ}r_i(s)\widetilde{m}_j(2^{b_i+\alpha} L)(f)\Big |^p ds\bigg)^{\frac1p}\bigg\|_{L^p(d\mu)}\\
  &\leq \bigg(\int_{0}^{1}\bigg\|\sup_{   |\alpha|\leq 4^{N+1}}  \Big|\sum_{i\in\ZZ}r_i(s)\widetilde{m}_j(2^{b_i+\alpha} L)(f)\Big|\   \bigg \|_{L^p(d\mu)}^p ds\bigg)^{\frac1p}.
\end{align*}
Following by  titling Lemma~\ref{Lemma6} and supp$\,\phi\subseteq \RR^+$, we have for any $\beta\in\NN$,
\begin{align*}
  &  \bigg \|\phi\sum_{i\in\ZZ} r_i(s)\widetilde{m}_j(2^{\beta+b_i+\alpha}\cdot)\bigg \|_{W^q_s}+ \bigg|\sum_{i\in\ZZ} r_i(s)\widetilde{m}_j(2^{b_i+\alpha}\times 0)\bigg|\\
  &=\bigg \|\phi\sum_{i\in\ZZ} \sum_{k\in F_j}r_i(s)  \left(\phi m(2^k\cdot)\right) (2^{\beta+b_i+\alpha-k}\cdot)\bigg\|_{W^q_s}
  \\
  &\leq C\sup_{k\in F_j}  \left\|\phi m(2^k\cdot)\right\|_{W^q_s}\leq C\omega^{*}(2^{2^{j-1}}).
\end{align*}
Following by Theorem~\ref{Theorem1}, for $s>n/r$ and $p>r$, we have
\begin{align*}
  &  \left\|M_{\widetilde{m}_j,L}^{\mathrm{dyad}}(f)\right\|_{L^p(d\mu)\rightarrow L^p(d\mu)}\\
  &\leq C2^{j/2}\,\Bigg(\sup_{\beta>0}  \bigg\|\phi\sum_{i\in\ZZ} r_i(s)\widetilde{m}_j(2^{\beta+b_i+\alpha}\cdot)\bigg\|_{W^q_s}+\bigg|\sum_{i\in\ZZ} r_i(s)\widetilde{m}_j(2^{b_i+\alpha}\times 0)\bigg| \, \Bigg)\\
  &\leq C2^{j/2}\omega^{*}(2^{2^{j-1}}).
\end{align*}
As $\omega^*$ is a non-increasing function, by a simple calculation, we can obtain that
\begin{align*}
 \left \|M_{m,L}^{\mathrm{dyad}}(f)\right\|_{L^p(d\mu)\rightarrow L^p(d\mu)}&\leq C\sum_{j\geq0}2^{j/2}\omega^{*}(2^{2^{j-1}})+C|m(0)|\\
 &\leq C\omega^{*}(0)+C\sum_{j\geq2}\omega^{*}(2^{2^{j-1}})\int_{2^{2^{j-2}}}^{2^{2^{j-1}}}   \frac1{t\sqrt{\ln{t}}}dt+C|m(0)|\\
   &\leq C\omega^{*}(0)+C\sum_{k\geq0}\int_{2^{2^{k}}}^{2^{2^{k+1}}}   \frac{\omega^{*}(t)}{t\sqrt{\ln{t}}}dt+C|m(0)|\\
  &\leq C\Big(\omega^{*}(0)+\sum_{\ell\geq2}\frac{\omega^{*}(\ell)}{\ell\sqrt{\log\ell}}+|m(0)|\Big)<\infty.
\end{align*}

\vspace{0.3cm}

\noindent{\rm{\bf {Proof of (ii)}}}. Set $\psi\in C_c^\infty(\RR)$ with support $\{\xi:1/4\leq |\xi|\leq 4\}$ and $\psi(\xi)=1$ on the support of $\phi$. Pick up a  radial cutoff function $\chi\in C_c^\infty(\RR)$ with support $\{\xi:1/8\leq |\xi|\leq 8\}$; moreover, $\sum_{\ell\in\ZZ}\chi(2^{-\ell}x)=1$ for $x\neq0$. Define $\chi_0(x)= 1-\sum_{\ell\geq0}\chi(2^{-\ell}x) $ and $\chi_\ell(x)=\chi(2^{-\ell}x)$ for $\ell>0$.
Define a sequence $h_k^{j,\ell}$ by
\begin{align*}
  h_k^{j,0}(x)=
  \begin{cases}
    \mathcal{F}\big(\phi m(2^k\cdot)\big)(x), & \mbox{if \ } |x|\leq 4 \mathrm{\ and \ } k\in F_j \\[4pt]
    0, &  \mbox{if\  } |x|>4 \mathrm{\ or\ } k\notin F_j.
  \end{cases}
\end{align*}
and\begin{align*}
  h_k^{j,\ell}(x)=
  \begin{cases}
    \mathcal{F}\big(\phi m(2^k\cdot)\big)(x), & \mbox{if\  }  |x|\in [2^{\ell-4},2^{\ell+4}] \mathrm{\ and \ } k\in F_j \\[4pt]
    0, &  \mbox{if\  }  |x|\notin [2^{\ell-4},2^{\ell+4}] \mathrm{\ or\ } k\notin F_j.
  \end{cases}
\end{align*}
Then
\begin{align*}
 \widetilde{m}_j(t)&=\sum_{k\in E_j}\sum_{\ell\geq0}\psi(2^{-k}t)   \left(\phi m(2^k\cdot))*\check{\chi}_\ell \right)(2^{-k}t)\\
 &=\sum_{\ell\geq0}\sum_{k\in E_j}\Big(\psi \mathcal{F}^{-1}(h_k^{j,\ell}\chi_\ell)\Big)(2^{-k}t).
\end{align*}
Define
$$\widetilde{m}_{j,\ell}(t):=\sum_{k\in E_j}\Big(\psi \mathcal{F}^{-1}(h_k^{j,\ell}\chi_\ell)\Big)(2^{-k}t).$$

Following a standard argument, we have
\begin{align}\label{IBJ}
  \nonumber \sup_{t>0}   \left|\widetilde{m}_{j,\ell}(tL)f(x)\right|& =\sup_{d \in\ZZ}\sup_{1\leq t\leq2}  \left|\widetilde{m}_{j,\ell}(2^d tL)f(x)\right|\\
  &\leq C\sup_{d\in\ZZ}   \left|\widetilde{m}_{j,\ell}(2^d L)f(x)\right |+C\sup_{d\in\ZZ}\left(\int_{1}^{2}   \left |  \left [\widetilde{m}_{j,\ell}(2^d tL)\right] f(x)  \right|^pdt\right)^{\frac{1}{pp'}}\\
 \nonumber &\times
  \left(\int_{1}^{2}  \Big |t\frac{\partial}{\partial t}  \left[\widetilde{m}_{j,\ell}(2^d tL)\right]f(x)\Big|^pdt\right)^{\frac{1}{p^2}}.
\end{align}
Then we can reduce the matters to the dyadic case. Then for any $t\in[1,2]$, there exists a uniformly constant $C$ such that for any $|\alpha|\leq 4^N$, $\beta\in\ZZ$,
\begin{align}\label{m1_e}
   \nonumber&  \bigg \|\phi\sum_{i\in\ZZ} r_i(s)\widetilde{m}_{j,\ell}(2^{\beta+b_i+\alpha}\cdot)\bigg\|_{W^q_{s-1/p}}+\bigg|\sum_{i\in\ZZ} r_i(s)\widetilde{m}_{j,\ell}(2^{b_i+\alpha}\times 0)\bigg|\\
 \nonumber &=\bigg\|\phi\sum_{i\in\ZZ} \sum_{k\in F_j}r_i(s)\Big(\psi \mathcal{F}^{-1}(h_k^{j,\ell}\chi_\ell)\Big)(2^{\beta+b_i+\alpha-k}\cdot)\bigg\|_{W^q_{s-1/p}}  \\
  &\leq C\sup_{k\in F_j}   \left\|\psi \mathcal{F}^{-1}(h_k^{j,\ell}\chi_\ell)\right\|_{W^q_{s-1/p}}.
\end{align}
For another term, indeed,
\begin{align*}
t\frac{\partial}{\partial t}   \left[\widetilde{m}_{j,\ell}(t\xi)\right]
&=\sum_{k\in F_j}\Big(2^{-k}t\xi \psi^{'}(2^{-k}t\xi) \mathcal{F}^{-1}(h_k^{j,\ell}\chi_\ell))(2^{-k}t\xi)\\
 &\ \ \ \ \  +2^{\ell}2^{-k}t\xi \psi(2^{-k}t\xi) \mathcal{F}^{-1}(h_k^{j,\ell}\widetilde{\chi}_\ell))(2^{-k}t\xi)\Big)\\
& :=\widetilde{\widetilde{m}}_{j,\ell}(t\xi),\end{align*}
where $\widetilde{\chi}(\lambda)=\lambda\chi(\lambda)$, and $\widetilde{\chi}_\ell(\lambda)=\widetilde{\chi}(2^{-\ell}\lambda)$ for $\ell\geq1$ and $\widetilde{\chi}_0(\lambda)=\lambda\chi_0(\lambda)$.
Then for any $t\in [1,2]$, $|\alpha|\leq 4^N$, $\beta\in\ZZ$,
\begin{align}\label{m2_e}
 \nonumber &\bigg\|\phi\sum_{i\in\ZZ} r_i(s) \widetilde{\widetilde{m}}_{j,\ell}(2^{\beta+b_i+\alpha}\cdot)\bigg\|_{W^q_{s-1/p}}+\bigg|\sum_{i\in\ZZ} r_i(s)\widetilde{\widetilde{m}}_{j,\ell}(2^{b_i+\alpha}\times 0)\bigg|\\
 \nonumber &\leq C\sup_{k\in F_j} \bigg\| \xi \psi'(\xi)\mathcal{F}^{-1}(h_k^{j,\ell}\chi_\ell)(\xi)+2^{\ell}\xi \psi(\xi) \mathcal{F}^{-1}(h_k^{j,\ell}\widetilde{\chi}_\ell)(\xi)\bigg\|_{W^q_{s-1/p}}\\
 &\leq C\sup_{k\in F_j}2^\ell\left(\left\|\widetilde{\psi} \mathcal{F}^{-1}(h_k^{j,\ell}\chi_\ell)\right\|_{W^q_{s-1/p}}
 +\left\|\widetilde{\psi} \mathcal{F}^{-1}(h_k^{j,\ell}\widetilde\chi_\ell)\right\|_{W^q_{s-1/p}}\right),
\end{align}
where $\widetilde{\psi}(\xi)\in C_c^\infty([1/8,8])$ and $\widetilde{\psi}=1$ on the support of $\psi$.


Discussing similarly as the proof of \rm{(i)}, and using Theorem~\ref{Theorem1}, estimates~\eqref{IBJ}--\eqref{m2_e} and the  H\"older inequality, for $h^{p-2}\in A_{p/r}(h^2d\mu)$ and $s-1/p>n/r$, we can deduce that
\begin{align}\label{mkjl_e}
 \nonumber &\Big \|\sup_{t>0} \left|\widetilde{m}_{j,\ell}(tL)f(x)\right|\Big  \|_{L^p(d\mu)}\\ \nonumber
  &\leq C\Big \| \sup_{d\in\ZZ}  \left|\widetilde{m}_{j,\ell}(2^d L)f(x)\right|  \Big \|_{L^p(d\mu)}
  + C\left(\int_{1}^{2}\Big \|\sup_{d\in\ZZ}  \left| \big[\widetilde{m}_{j,\ell}(2^d tL)\big]f(x)\right |  \Big \|_{L^p(d\mu)}^pdt\right)^{\frac{1}{pp'}}\\\nonumber
  &\ \ \ \ \times \left(\int_{1}^{2}   \Big \|\sup_{d\in\ZZ}  \big |  \big [\widetilde{\widetilde{m}}_{j,\ell}(2^d tL)\big] f(x)\big |\, \Big\|_{L^p(d\mu)}^p dt\right)^{\frac{1}{p^2}}\\
  &\leq C2^{\frac{j}2} \sup_{k\in F_j}2^{\frac\ell{p}}  \left(\left\|\psi \mathcal{F}^{-1}(h_k^{j,\ell}\chi_\ell)\right\|_{W^q_{s-1/p-\varepsilon}}
  +\left\|\widetilde\psi \mathcal{F}^{-1}(h_k^{j,\ell}\widetilde\chi_\ell)\right\|_{W^q_{s-1/p-\varepsilon}}\right).
\end{align}

\noindent Let $\eta\in C_c^\infty([1/16,16])$ and $\eta=1$ on the support of $\chi$. Noting that if $k\in F_j$, for any $0<\varepsilon<s-1/p-n/r$, there exists a constant $C=C_\varepsilon$ such that
\begin{align}\label{hkjl_1}
  \nonumber &  \left \|\psi \mathcal{F}^{-1}(h_k^{j,\ell}\chi_\ell)  \right \|_{W^q_{s-\frac1p-\varepsilon}} \leq   \left\| \mathcal{F}^{-1}(h_k^{j,\ell}\chi_\ell)\right\|_{W^q_{s-\frac1p-\varepsilon}} \\
  \nonumber &= \left\|\mathcal{F}^{-1}  \left[(1+|\cdot|)^{-\frac1p-\varepsilon}\eta_\ell\right]* \mathcal{F}^{-1}  \left[(1+|\cdot|)^sh^{j,\ell}_k\chi_\ell\right] \ \right\|_{L^q} \\ \nonumber
  &\leq 2^{-\frac{\ell}{p}-\varepsilon\ell} \left\|\mathcal{F}^{-1}  \left[2^{\frac{\ell}{p}+\varepsilon\ell}(1+|\cdot|)^{-\frac1p-\varepsilon}\eta_\ell \right]\ \right\|_{L^1}
\left \|\mathcal{F}^{-1}  \left[(1+|\cdot|)^s h_k^{j,\ell}\chi_\ell\right]\ \right\|_{L^q}\\
  & \leq C2^{-\frac{\ell}{p}-\varepsilon\ell}  \left\|\mathcal{F}^{-1}(h_k^{j,\ell}\chi_\ell)\right\|_{W^q_s}.
\end{align}
From \eqref{hkjl_1}, Young's inequality and  \eqref{eq1.4.2}, we have
\begin{align}\label{hkjl1_1_1}
\nonumber\sup_{k\in F_j}2^{\frac\ell{p}} \left\|\psi \mathcal{F}^{-1}(h_k^{j,\ell}\chi_\ell)\right\|_{W^q_{s-1/p-\varepsilon}}&
   \leq 2^{-\varepsilon\ell}\sup_{k\in F_j} \left\|\mathcal{F}^{-1}(h_k^{j,\ell}\chi_\ell)\right\|_{W^q_s}\\
   &\leq 2^{-\varepsilon\ell}\sup_{k\in F_j} \left\|\phi m(2^k\cdot)\right\|_{W^q_s} \leq 2^{-\varepsilon\ell}\omega^{*}(2^{2^{j-1}}).
\end{align}
Similarly, we can estimate $\left\|\widetilde\psi \mathcal{F}^{-1}(h_k^{j,\ell}\widetilde\chi_\ell)\right\|_{W^q_{s-1/p-\varepsilon}}$.  Combining \eqref{mkjl_e} and \eqref{hkjl1_1_1}, we have
\begin{align*}
  &\Big\|\sup_{t>0} \big|\widetilde{m}_{j,\ell}(tL)f(x)\big|\ \Big\|_{L^p(d\mu)}
  \leq C2^{-\varepsilon \ell}2^{j/2}\omega^{*}(2^{2^{j-1}}).
\end{align*}
Summing $j$ and $\ell$, we obtain that $M_{m,L}$ is bounded on $L^p(d\mu)$ if $s>n/r+1/p$.

 This completes the proof of Theorem~\ref{Theorem3}.  \hfill $\Box$

 \medskip

There is an immediate consequence of Theorem~\ref{Theorem3}.

\begin{corollary}\label{Theorem6}
 Suppose that   $L$ is a self-adjoint operator that satisfies  Gaussian estimate  \eqref{ge} and Plancherel condition \eqref{e1.4} for some $q\in[2,\infty]$. Assume there exists an $L$-harmonic $h$ such that conditions  \rm{(H-1)} and \eqref{LUGh} hold.
Suppose   $2\leq p<\infty$ and $h^{p-2}\in A_{p/2}(h^2d\mu)$. For some $s>0$, the bounded Borel function $m$ satisfies
\begin{align*}
    \|\phi m(2^k\cdot)\|_{W^q_{s}}\leq \omega(k).
\end{align*}
Assume that the non-increasing rearrangement function $\omega^{*}$ satisfies
\begin{align*}
    \omega^{*}(0)+\sum_{\ell\geq2}\frac{\omega^{*}(\ell)}{\ell\sqrt{\log\ell}}<\infty.
\end{align*}
(i) Assume $s>n/2$, then $M_{m,L}^{\mathrm{dyad}}$ is bounded on $L^p(d\mu)$.

\noindent (ii) Assume $s>n/2+1/p$, then $M_{m,L}$ is bounded on $L^p(d\mu)$.

\end{corollary}

\begin{proof}
(i)  For $s>n/2$. There exist an $r\in [1,2)$ such that $s>n/r$. If $h^{p-2}\in A_{p/2}(h^2d\mu)$, then $h^{p-2}\in A_{p/r}(h^2d\mu)$.
From Theorem~\ref{Theorem3}, we obtain that $M_{m,L}^{\mathrm{dyad}}$ is bounded on $L^p(d\mu)$.

(ii) Let $r=(\frac12+\frac{1}{pn})^{-1}$. We have $r\in [1,2)$, $s>n/r$ and $A_{p/2}(h^2d\mu)\subseteq A_{p/r}(h^2d\mu)$. If $h^{p-2}\in A_{p/2}(h^2d\mu)$, then $h^{p-2}\in A_{p/r}(h^2d\mu)$.
From Theorem~\ref{Theorem3}, we obtain that $M_{m,L}$ is bounded on $L^p(d\mu)$.
\end{proof}

\vspace{0.2cm}

\subsection{Proof of Theorem~\ref{Theorem8}}  \label{subsec:6.2}

The proof of Theorem~\ref{Theorem8} is based on the following proposition, whose  proof will be given later.

\begin{proposition}\label{H1toL1}
 Suppose that   $L$ is a self-adjoint operator that satisfies   Gaussian estimate  \eqref{ge} and Plancherel condition \eqref{e1.4} for some $q\in[2,\infty]$. Suppose $1<p<\infty$.
Let $m$ be a  bounded Borel function with $m(0)=0$.
$M_{m,L}$ is bounded on $L^2(X)$ with norm $A_0$. If $m$ satisfies
\begin{align}\label{pro6.3e00}
  \sup_{k\in\ZZ} \left\|\phi m(2^k\cdot)\right\|_{W^2_\alpha}\leq A_1,
\end{align}
where $\alpha>n/2+1/q'$.  Then $M_{m,L}$ is bounded from $L^1(d\mu)$ to $L^{1,\infty}(d\mu)$  with norm $C(A_0+A_1)$.
\end{proposition}

\begin{proof}[Proof of Theorem~\ref{Theorem8}]
If $m$ satisfies \eqref{tm7c1} for some $s>n/2+1/q'$, we apply Theorem \ref{Theorem3} or Corollary~\ref{Theorem6} to get $M_{m,L}$ is bounded on $L^2(d\mu)$.  Then Proposition~\ref{H1toL1} implies that  $M_{m,L}$ is bounded from $L^1(d\mu)$ to $L^{1,\infty}(d\mu)$. By interpolation theorem,
$M_{m,L}$ is bounded on $L^p(d\mu)$ for all $1<p\leq 2$. This completes the proof of Theorem~\ref{Theorem8}.
\end{proof}

\medskip

To prove Proposition~\ref{H1toL1}, we need
the following result for a kind of Calder\'on-Zygmund singular integrals (see \cite{A2007, DM1999} for more reference).
Define an $L^2$-bounded operator: for $f\in L^2(d\mu)$ and $t>0$,
$$T_{t}(L)f(x)=\int_X K_{T_{t}(L)}(x,y)f(y)d\mu(y).$$
\begin{lemma}\label{SGL2}
 Suppose that   $L$ is a self-adjoint operator that satisfies   Gaussian estimate  \eqref{ge} and Plancherel condition \eqref{e1.4} for some $q\in[2,\infty]$. Suppose $1<p<\infty$.
 Let $T^*(f)=\sup_{t>0}|T_{t}(L)f|$. $T^*$ is bounded on $L^{2}(d\mu)$, i.e. there exists a constant $B_0$  such that
$$
\|T^*(f)  \, \|_{L^2(d\mu)}\leq B_0\|f\|_{L^2(d\mu)}.
$$
If $T^*$ satisfies
\begin{align}\label{horweak}
  \sup_{y\in X} \sup_{r>0}\int_{X-B(y,2r)}\sup_{t>0} \left|K_{T_{t}(L)(I-e^{-r^2L})}(x,y)\right|d\mu(x)<B_1.
\end{align}
Then $T^*$ is bounded from $L^1(d\mu)$ to $L^{1,\infty}(d\mu)$ with bound $C(B_0+B_1)$.
\end{lemma}

\begin{proof}
The proof of this result essentially can be found in
\cite[Theorem~1]{DM1999} or \cite[Theorem~2.1]{A2007}. For the completeness of the paper, we would present the proof in detail.

Given $f\in \mathcal{S}$ and $\alpha>\|f\|_{L^1(d\mu)}\mu(X)^{-1}$, following by the Calder\'on-Zygmund decomposition, we obtain a decomposition $f=g+b$, where $b=\sum_{i\in\ZZ}b_i$ and $b_i$ support in $Q_i$. The `good' function $g$ can be treated by using the $L^2$-boundedness of $T^*$. Indeed,
\begin{align*}
\mu\left(\left\{ x\in X: T^*(g)(x)>\alpha   \right\}\right)\leq C\alpha^{-2}\|T^*(g)\|_{L^2(d\mu)}^2\leq CB_0\alpha^{-1}\|f\|_{L^1(d\mu)} .
\end{align*}
While for the `bad' function $b$,
\begin{align*}
&\mu\left(\left\{ x\in X: T^*(b)(x)>\alpha   \right\}\right)\\[2pt]
&\leq \mu\Big( \bigcup_j Q_j^*   \Big)+\mu\Big(\Big\{ x\in \Big(\bigcup_j Q_j^*\Big)^c : T^*(b)(x)>\alpha   \Big\}\Big).
\end{align*}
The the treatment of  the  part $\mu( \bigcup_j Q_j^*)$ is the standard procedure, it suffices to show
\begin{align}\label{le1es}
   \mu\left(\left\{ x\in \Big(\bigcup_j Q_j^*\Big)^c : T^*(b)(x)>\alpha   \right\}\right)\leq C\alpha^{-1}\|f\|_{L^1(d\mu)} ,
\end{align}
where $Q^*_j=2Q_j$.

\noindent For  $b_i$ with support $Q_i$ with radius $r_i$, let $t_i=r_i^2$. Then
$$
   T^*\Big(\sum_ib_i \Big)(x)\leq T^*\Big(\sum_i\exp(-t_iL)b_i\Big)+T^*\Big(\sum_i(I-\exp(-t_iL))b_i\Big).
$$
Similar to the discussion of \cite[p.241~(10)]{DM1999}, we have
\begin{align*}
   \Big \|\sum_i \exp(-t_iL)b_i \Big\|_{L^{2}(d\mu)}^2 \leq C\alpha \|f\|_{L^1(d\mu)},
\end{align*}
which combines with the $L^2$-boundedness of $T^*$  implies that
\begin{align*}
    &\mu\bigg(\Big\{ x\in \Big(\bigcup_j Q_j^*\Big)^c : T^*\Big(\sum_i \exp(-t_iL)b_i \Big)(x)>\alpha   \Big\}\bigg)\\
    &\leq C\alpha^{-2}\Big\|T^*\Big(\sum_i \exp(-t_iL)b_i\Big)\,\Big\|_{L^{2}(d\mu)}^{2}\\
    &\leq CB_0\alpha^{-2}\Big\|\sum_i \exp(-t_iL)b_i\Big\|_{L^{2}(d\mu)}^{2}\\
    &\leq CB_0\alpha^{-1}\|f\|_{L^1(d\mu)}.
\end{align*}
For the second term, with the condition~\eqref{horweak}, we have
\begin{align*}
    &\mu\left(\left\{ x\in \Big(\bigcup_j Q_j^*\Big)^c : T^*\Big(\sum_i \left(I-\exp(-t_iL)\right)b_i)\Big)(x)>\alpha   \right\}\right)\\
    &\leq C\alpha^{-1}\int_{(\bigcup_j Q_j^*)^c}\Big|T^* \Big(\sum_i \left(I-\exp(-t_iL)\right)b_i\Big)(x)\Big|d\mu(x)\\
    &=C\alpha^{-1}\int_{(\bigcup_j Q_j^*)^c}\sup_{s>0}   \bigg|\int_X  K_{T_{s}(L)}(x,y)\sum_i  \left (I-\exp(-t_iL)\right) b_i(y)d\mu(y)\bigg|d\mu(x)\\
    &\leq C\sum_{i}\alpha^{-1}\int_{(Q_i^*)^c}\int_{Q_i}\sup_{s>0}  \left |K_{T_{s}(L)(I-\exp(-t_iL))}(x,y)b_i(y)\right| d\mu(y)d\mu(x)\\
    &\leq C\sum_{i}\alpha^{-1}\int_{Q_i}|b_i(y)|\int_{X\setminus B(y,r_i)}\sup_{s>0}   \left|K_{T_{s}(L)(I-\exp(-t_iL))}(x,y)\right| d\mu(x)d\mu(y)\\
    &\leq CB_1\alpha^{-1}\|f\|_{L^1(d\mu)}.
\end{align*}
This gives \eqref{le1es}, which combines with the  estimate for `good' function $g$ and $\sum_j\mu(Q_j)\leq C\alpha^{-1}\|f\|_{L^1(d\mu)}$,  shows the weak type bound of $T^*$.

This finishes the proof of Lemma~\ref{SGL2}.
\end{proof}

\smallskip

\begin{proof}[Proof of Proposition~\ref{H1toL1}]
	Recall that  the fractional derivatives $F^{(\alpha)}$ of function $F$ of order $\alpha$   is defined by
\eqref{eyyy}.
 Let $G\in  C_c^\infty((0,\infty))$ with $|\partial^k G(\xi)|\leq 1$ for any $0\leq k\leq n+2$, and $G(\xi)=1$ if $\xi\leq 1$ and $G(\xi)=0$ if $\xi\geq2$. Write
$\eta_{(\epsilon)}(\xi)=G(\epsilon \xi)-G(2\epsilon^{-1}\xi)$. Write  $m_{r,t,\epsilon}(\xi)=m(t\xi)\eta_{(\epsilon)}(\xi)(I-e^{-r^2\xi})$.
    Note that from \cite[(3.1)]{GT1979}, we have the following formula: for $\delta>1/2$,
\begin{align*}
    m_{r,t,\epsilon}(\xi) &=C_\delta\int_{0}^{\infty} s^{\delta}  m_{r,t,\epsilon}^{(\delta)}(s)
    \frac{\xi}{s}  \Big(1-\frac{\xi}{s}\Big)_+^{\delta-1}\frac{ds}{s}.
\end{align*}
 By spectral theory, we have
\begin{align*}
    m_{r,t,\epsilon}(L) &=C_\delta\int_{0}^{\infty} s^{\delta}  m_{r,t,\epsilon}^{(\delta)}(s)
    \frac{L}{s} \Big(1-\frac{L}{s}\Big)_+^{\delta-1}\frac{ds}{s},
\end{align*}
and so
\begin{align*}
    \left|K_{m_{r,t,\epsilon}(L)}(x,y)\right| & \leq \sum_{j\in\ZZ}C_\delta\Big|\int_{e^j}^{e^{j+1}}s^{\delta}  m_{r,t,\epsilon}^{(\delta)}(s)
    K_{\frac{L}{s}(1-\frac{L}{s})_+^{\delta-1}}(x,y)\frac{ds}{s}\Big| .
\end{align*}
From Minkowski's inequality and the Cauchy-Schwarz inequality,
\begin{align}\label{pro6.3e0}
    &\int_{X\setminus B(y,2r)} \sup_{0<\epsilon<1} \sup_{t>0} \left|K_{m_{r,t,\epsilon}(L)}(x,y)\right|d\mu(x)\\
 \nonumber   &\leq C_\delta\sum_{j\in\ZZ} \sup_{0<\epsilon<1}\sup_{t>0}\left(\int_{e^j}^{e^{j+1}} \left|s^{\delta}  m_{r,t,\epsilon}^{(\delta)}(s)\right|^2 \frac{ds}{s} \right)^{\frac12}\\
 \nonumber  &\ \ \ \ \  \times \int_{X\setminus B(y,2r)}  \left(\int_{e^j}^{e^{j+1}} \big| K_{\frac{L}{s}(1-\frac{L}{s})_+^{\delta-1}}(x,y)\big|^2   \frac{ds}{s} \right)^{\frac12} d\mu(x)  .
\end{align}
Let $\Phi(t)\in C_c^\infty(\RR)$ be a monotone decreasing function for $t\geq 0$, with
\begin{align*}
&  \Phi(t) =\begin{cases}
             1, &   0\leq t\leq\frac13, \\[6pt]
             0, &  t\geq \frac23.
           \end{cases}
&   \Phi(t) =\begin{cases}
             1, &   -\frac13\leq t\leq0, \\
             1-\Phi(t+1), & -\frac23\leq t\leq -\frac13 , \\
             0, &  t\leq -\frac23.
           \end{cases}
\end{align*}
Then we have $\sum_{j\in\ZZ}\Phi(t-j)=1$ for all $t\in\RR$. Let $\varphi^2(t-j)=\Phi(t-j)$. Then $\varphi(t-j)\in C_c^\infty(\RR)$    is a deceasing function for $t\geq j$ and $\sum_{j\in\ZZ}\varphi^2(t-j)=1$.
From \cite[p\,263-p\,266]{GT1979}  for $\delta>1/2$,
\begin{align}\label{pro6.3e1}
&\left(\int_{e^j}^{e^{j+1}} \left|s^{\delta}  m_{r,t,\epsilon}^{(\delta)}(s)\right|^2 \frac{ds}{s} \right)^{\frac12}\leq C_\delta\sum_{k\geq j} e^{(j-k)\delta}\|\varphi^2(\cdot-k)m_{r,t,\epsilon}(e^{\cdot})\|_{W_\delta^2}.
\end{align}
Let $\phi\in C_c^\infty([1/4,1])$ and $\sum_{\ell\in\ZZ}\phi(2^{-\ell}t)=1$ for all $t>0$. Obviously,
\begin{align}\label{pro6.3e2}
\nonumber &\sup_{0<\epsilon<1} \sup_{t>0}\|\varphi^2(\cdot-k)m_{r,t,\epsilon}(e^{\cdot})\|_{W_\delta^2}\\
 \nonumber&\leq C\sup_{0<\epsilon<1}\sup_{t>0}\|\varphi m(te^{\cdot+k})\|_{W_\delta^2}\times\|\varphi(1-e^{-r^2e^{\cdot+k}})\eta_{(\epsilon)}(e^{\cdot+k})\|_{C^{[\delta]+1}}\\
  &
  \leq C\min\{1,e^kr^2\}\sup_{\ell\in\ZZ}\|\phi m(2^\ell\cdot)\|_{W_\delta^2},
\end{align}
where the constant $C$ is independent of $\epsilon$.

Let $F_s^\delta(\xi)=\frac{\xi}{s}(1-\frac{\xi}{s})_+^{\delta-1}$. Observe that for any $  \xi>0$,
$$F_s^\delta(\xi^2)=\sum_{\ell\in\ZZ}\phi\Big(\frac{\xi}{2^\ell s^{\frac12}}\Big)F_s^\delta(\xi^2)=\sum_{\ell=-\infty}^{1}\phi\Big(\frac{\xi}{2^\ell s^{\frac12}}\Big)F_s^\delta(\xi^2).$$
Write  $F_{s,\ell}^\delta(\xi)=\phi(\frac{\xi}{2^\ell s^{1/2}})F_s^\delta(\xi^2)$. From  \cite[Lemmas~4.3 and  4.4]{DOS2002} and the Cauchy-Schwarz inequality,   we see that for $\beta>n/2$ and $2\leq q\leq \infty$,
\begin{align}\label{pro6.3e3}
 \nonumber    &\int_{X\setminus  B(y,2r)}  \left(\int_{e^{j}}^{e^{j+1}} \big| K_{F_{s,\ell}^\delta(\sqrt{L})}(x,y)\big|^2   \frac{ds}{s} \right)^{\frac12}d\mu(x)\\
   \nonumber  &\leq C\sum_{\ell=-\infty}^{1} \left( \int_{e^{j}}^{e^{j+1}} \int_{X} \big| K_{F_{s,\ell}^\delta(\sqrt{L})}(x,y)\big|^2 \big(1+2^{\ell}e^{\frac{j}2}d(x,y)\big)^{2\beta}  d\mu(x)\frac{ds}{s} \right)^{\frac12}\\
  \nonumber  &\ \ \ \   \times  \left(\int_{X\setminus B(y,2r)}  \big(1+2^{\ell}e^{\frac{j}2}d(x,y)\big)^{-2\beta}    d\mu(x)\right)^{\frac12}\\
  &\leq C\sum_{\ell=-\infty}^{1}\left(1+2^{\ell}e^{\frac{j}2}r\right)^{-\beta+\frac{n}{2}}\times
  \left(\int_{e^{j}}^{e^{j+1}} \left\|\phi F_{s,\ell}^\delta(2^\ell s^{\frac12}\cdot)\right\|_{W^q_{\beta+\varepsilon}}^2  \frac{ds}{s}\right)^{\frac12}.
\end{align}
From \cite[Lemma~2.2]{CDY2013}, we see  that for $\delta-1>\beta-1/q$,
\begin{align*}
\left\|\phi F_{s,\ell}^\delta(2^\ell s^{\frac12}\cdot)\right\|_{W^q_{\beta+\varepsilon}}&\leq C 2^{2\ell}\sup_{\ell\in\ZZ:\ell\leq1}\|\phi (1-2^{2\ell}\xi^2)^{\delta-1}_+\|_{W^q_{\beta+\varepsilon}}
\leq C2^{2\ell}
\end{align*}
and
\begin{align}\label{pro6.3e4}
	\sum_{\ell=-\infty}^1 (1+2^{\ell}e^{\frac{j}2}r)^{-\beta+\frac{n}{2}}2^{2\ell}\leq C(1+e^{\frac{j}2}r)^{-\min\{2,(\beta-\frac{n}2)\}}.
\end{align}
Now we apply  estimates~\eqref{pro6.3e0}-\eqref{pro6.3e4} and   the condition~\eqref{pro6.3e00}
to obtain
\begin{align*}
    &\int_{X\setminus B(y,2r)}\sup_{0<\epsilon<1} \sup_{t>0} \left|K_{m_{r,t,\epsilon}(L)}(x,y)\right|d\mu(x)\\
    &\leq C_\delta\sup_{\ell\in\ZZ}\|\phi m(2^\ell\cdot)\|_{W_\delta^2}\sum_{j\in\ZZ} \sum_{k\geq j}e^{(j-k)\delta} \min\{1,e^kr^2\}(1+e^jr^2)^{-\min\{1,\frac12(\beta-\frac{n}2)\}}\\
    &\leq C_\delta A_1.
\end{align*}
The operator $\eta_{(\epsilon)}(L)$ is bounded on $L^2(d\mu)$ uniformly in $\epsilon$ and $M_{m,L}$ is bounded on $L^2(d\mu)$.
Then there exists a constant $C$ which is independent of $\epsilon$ such that
$$ \Big\|\sup_{t>0} \left|m(tL)\eta_{(\epsilon)}(L)f\right|\Big \|_{L^2(d\mu)}\leq C\|f\|_{L^2(d\mu)}.$$
It follows from Lemma~\ref{SGL2}  that there exists a positive constant $C$ which is independent of $\epsilon$  such that
 \begin{align*}
     \Big\|\sup_{t>0} \left|m(tL)\eta_{(\epsilon)}(L)f\right|\Big \|_{L^{1,\infty}(d\mu)}\leq C(A_0+ A_1) \left\|f\right\|_{L^1(d\mu)}.
 \end{align*}
On other hand, from the $L^2$-boundedness of $M_{m,L}$ and the Littlewood-Paley theory, we have
 \begin{align*}
    \lim_{\epsilon\rightarrow 0+} \Big\|\sup_{t>0} \left|m(tL)(\eta_{(\epsilon)}(L)-I)f\right|\ \Big\|_{L^2(d\mu)}=0.
 \end{align*}
Then there exists a sub-sequence $\epsilon_k\rightarrow 0$ and a set of zero-measure $D$ such that
$$\lim_{k\rightarrow \infty}\sup_{t>0} \left|m(tL)(\eta_{(\epsilon_k)}(L)-I)f(x)\right|=0,\ \ \forall \,x\in X-D.$$
From Fatou's lemma, Chebyshev’s inequality and the above estimates, for any $f\in L^1$, we have
 \begin{align*}
     \Big\|\sup_{t>0} \left|m(tL)f\right|\Big \|_{L^{1,\infty}(d\mu)}&\leq C\liminf_{k\rightarrow \infty}\|\sup_{t>0} |m(tL)\eta_{(\epsilon_k)}(L)f|\,\|_{L^{1,\infty}(d\mu)}\\
    &+C\sup_{\lambda>0}\lambda \mu\Big(\Big\{x\in X:\lim_{k\rightarrow \infty}\sup_{t>0}\Big|m(tL)(\eta_{(\epsilon_k)}(L)-I)f(x)\Big|>\frac\lambda2\Big\}\Big)\\
    &\leq C(A_0+A_1)\|f\|_{L^1(d\mu)}+C\|\sup_{t>0}|m(tL)(\eta_{(\epsilon_k)}(L)-I)f|\,\|_{L^{1}(D,d\mu)}\\
     &\leq C(A_0+ A_1) \left\|f\right\|_{L^1(d\mu)}.
 \end{align*}
 This finishes the proof of Proposition~\ref{H1toL1}.
\end{proof}

\medskip

\section{Applications}\label{sec:7}
\setcounter{equation}{0}
In this section, we give several examples that satisfy the assumptions in this paper.

\subsection{Schr\"odinger operators with inverse-square potential}  \label{subsec:7.1}
Let $X=\RR^n,n\geq3$ with Lebesgue measure $dx$. Consider
$$
    L_V=-\Delta+V(x),
$$
where $V(x)=\gamma|x|^{-2}$ with $\gamma>0$. Consider the function
$$
    h_V(x)=|x|^\tau,
$$
where $\tau=\frac12(\sqrt{(n-2)^2+4\gamma}-(n-2))>0.$

Gaussian bound for  the heat kernel $T_t(x,y)$ of the semigroup $\exp(-tL_V)$ and the Plancherel condition for $q=2$ follow from \cite[III.3]{COSY2016}. From \cite[Theorem~1.2]{IKO2017}, $h_V$ satisfies condition~\eqref{LUGh}.

\begin{lemma}\label{Lemma11}
Let $1\leq r<2, r<p<\infty$, then $h_V^{p-2}\in A_{p/r}(d\upsilon)$, where $d\upsilon=h_V^2dx$. If $p\geq2$, then $h_V^{p-2}\in A_{p/2}(d\upsilon)$.
\end{lemma}

\begin{proof}
Let $\mu_\alpha=|x|^\alpha dx$ for $\alpha>0$. Then
$$
    \mu_\alpha(B(x,t))\approx t^n(|x|+t)^\alpha,
$$
which implies
\begin{align*}
  &\left(\int_{B(x,t)}h_V^{p-2}h_V^2dx\right)\left(\int_{B(x,t)}h_V^{-\frac{p-2}{p/r-1}}h_V^2dx\right)^{p/r-1} \\
  &\approx t^{np/r}(|x|+t)^{p\tau }(|x|+t)^{p(2/r-1)\tau}\approx \left(\int_{B(x,t)}h_V^2dx\right)^{p/r},
\end{align*}
which, by the definition of Muckenhoupt weights, means that $h_V^{p-2}\in A_{p/r}(d\upsilon)$.

Assume $p>2$, then
\begin{align*}
  &\left(\int_{B(x,t)}h_V^{p-2}h_V^2dx\right)\left(\int_{B(x,t)}h_V^{-\frac{p-2}{p/2-1}}h_V^2dx\right)^{p/2-1} \\
  &\approx t^{np/2}(|x|+t)^{p\tau }\approx \left(\int_{B(x,t)}h_V^2dx\right)^{p/2},
\end{align*}
which means that $h_V^{p-2}\in A_{p/2}(d\upsilon)$ for $p>2$.  $p=2$ is obvious.

This ends the proof of Lemma~\ref{Lemma11}.
\end{proof}

If there exists a function $h$ such that condition~\eqref{LUGh}  is true, then from \cite[Propostion~2.3]{PSY2022}, there  exists a function $\varphi$ with $0<C^{-1}\leq\varphi\leq C<\infty$ such that conditions~\rm{(H-1)} and \eqref{LUGh}  hold for $\widetilde{h}=\varphi h$. Therefore, the assumptions of the paper are all satisfied.

\subsection{Scattering operators}  \label{subsec:7.2}

Let $X=\RR^3$ with the Lebesgue measure $d\mu$.  Let
$$L_S=-\Delta_3+V(x)=-(\partial_1^2+\partial_2^2+\partial_3^2)+V(x),$$
where $V(x)\geq 0$ is a compactly support function and satisfies
\begin{align}\label{So1}
  \frac{1}{4\pi} \sup_{x\in\RR^3}\int_{\RR^3}\frac{V(y)}{|x-y|}dy<1,
\end{align}
and also
\begin{align*}
  \int_{\RR^6} \frac{V(x)V(y)}{|x-y|^2}dxdy<\infty.
\end{align*}
From \cite[Theorem~7.15]{DOS2002}, the operator $L$ satisfies the Plancherel condition for $q=2$. From \cite{S1997}, condition \eqref{So1} implies that the heat kernel $T_t(x,y)$ of the semigroup $\exp(-tL_S)$ satisfies the following upper and lower Gaussian bounds
\begin{align}\label{LUG}
    \frac{C^{-1}}{\mu(B(x,\sqrt{t}))}\exp\left(-\frac{d(x,y)^2}{c_1t}\right)\leq T_t(x,y)\leq \frac{C}{\mu(B(x,\sqrt{t}))}\exp\left(-\frac{d(x,y)^2}{c_2t}\right).
\end{align}
From \cite[Proposition~3]{DP2017}, there exist a function $h$ and constant $C$ with
\begin{align}\label{eexx3}
    0<C^{-1}\leq h(x)\leq C<\infty
\end{align}
 such that for any $t>0$,
\begin{align}\label{eexx4}
    \exp(-tL_s)(h)(x)=h(x),\ \ \ \mathrm{a.e.} \   x\in \RR^3,
\end{align}
which combines \eqref{LUG} implies conditions~\rm{(H-1)}, \eqref{LUGh} hold. Meanwhile, it's apparently that $h^{p-2}\in A_{p/r}(h^2dx)$ for any $1\leq r<2$ and $p\in (r,\infty)$.

\subsection{Bessel operators}  \label{subsec:7.4}

  Let the space be $X=(0,\infty)$, the Euclidean space be $d(x,y)=|x-y|$ and the measure be $d\mu=x^\alpha dx$ with $\alpha>-1$. The Bessel operator is defined by
$$
    L_B=-f^{''}(x)-\frac{\alpha}{x}f'(x),\ \ x>0.
$$
From \cite[eq.~(7.4)]{DP2017}, the  heat kernel $T_t(x,y)$ of the semigroup $\exp(-tL_B)$ satisfies  the upper and lower Gaussian bounds  \eqref{LUG}. The Plancherel condition for $q=\infty$ is also satisfied.


\subsection{Laplace-Beltrami operators}  \label{subsec:7.5}

Let $(X,d,\mu)$ be a complete Riemannian manifold, $d$ be the related Riemannian distance and $\mu$  be the Riemannian measure. Assume that $(X,d,\mu)$ satisfies doubling property and the following Poincar$\mathrm{\acute{e}}$ inequality, for any $r>0$
$$
    \int_{B(x,r)}|f-f_B|^2 d\mu\leq Cr^2\int_{B(x,2r)}|\nabla f|^2d\mu,
$$
where $f_B$ denote the integral mean of $f$ on $B$ and $\nabla f$ denote the gradient of $f$.
It is well known that the heat kernel in this manifold satisfies the upper and lower Gaussian bounds  \eqref{LUG}, which can deduce that the related
Laplace-Beltrami operator satisfies the Plancherel condition for $q=\infty$. Similarly to the case of Scattering operators on $\RR^3$, the assumptions of Theorem~\ref{Theorem1} follow from \eqref{LUG}, \eqref{eexx3} and \eqref{eexx4}.


\bigskip

\noindent
{\bf Acknowledgements}: The authors    were supported  by National Key R$\&$D Program of China 2022YFA1005700. P. Chen was supported by NNSF of China 12171489. X. Lin was supported by the Fundamental Research Funds for the Central Universities, Sun Yat-sen University.



%

\end{document}